\documentclass[12pt]{amsart}

\usepackage{amsfonts,amssymb}
\usepackage{hyperref, graphicx}
\usepackage{epsfig}
\usepackage{latexsym}
\usepackage{amsmath,amsthm}
\usepackage{mathrsfs}
\usepackage[all,cmtip]{xy}
\usepackage{stmaryrd}

\theoremstyle{plain}
\newtheorem{theorem}{Theorem}[section]
\newtheorem{lemma}[theorem]{Lemma}
\newtheorem{definition}[theorem]{Definition}
\newtheorem{proposition}[theorem]{Proposition}

\newtheorem{remark}[theorem]{Remark}

\numberwithin{equation}{section}

\newcommand{\ra}{\rightarrow}
\newcommand{\Z}{\mathbb{Z}}

\newcommand{\R}{\mathbb{R}}

\newcommand{\sheaf}[1]{\mathcal{#1}}
\newcommand{\psheaf}[1]{\mathcal{#1}_\bullet}
\newcommand{\pchi}[1]{\chi_\mathrm{par}(#1)}
\newcommand{\pH}[1]{\mathrm{par}\text{-}P_{#1}}
\newcommand{\rphilb}[1]{\mathrm{par}\text{-}p_{#1}}
\newcommand{\pdeg}[1]{\mathrm{p}\text{-}\deg(#1)}
\newcommand{\pslope}[1]{\mu_{\mathrm{par}}(#1)}
\newcommand{\id}{\mathbf{1}}

\newcommand{\Hom}[3][]{\mathrm{Hom}_{#1}(#2, #3)}

\newcommand{\struct}[1]{\ensuremath{\mathcal{O}_{#1}}}
\newcommand{\Coh}{\mathbf{Coh}(X)}
\newcommand{\chCoh}{\mathbf{Coh}^\mathrm{ch}(X)}
\newcommand{\nCoh}{\mathbf{Coh}^{n\text{-}\mathrm{reg}}(X)}
\newcommand{\KM}{A\text{-}\mathbf{mod}}
\newcommand{\nKM}{A\text{-}\mathbf{mod}^{n\text{-}\mathrm{reg}}}
\newcommand{\Sch}{\mathbf{Sch}}
\newcommand{\Set}{\mathbf{Set}}
\newcommand{\PSh}{\mathbf{Coh}^{\mathrm{par}}(X, D)}
\newcommand{\RPSh}{\mathbf{Coh}^{\mathrm{\R\text{-}fil}}(X, D)}
\newcommand{\PKM}{A\text{-}\mathbf{mod}^{\mathrm{par}}_D}
\newcommand{\RPKM}{A\text{-}\mathbf{mod}^{\mathrm{\R\text{-}fil}}_{D}}
\newcommand{\FKM}{A\text{-}\mathbf{mod}^{\mathrm{fil}}}
\newcommand{\CKM}{A\text{-}\mathbf{mod}^{\mathrm{ch}}}
\newcommand{\RepKAl}{\mathrm{Rep}(\KAl)}
\newcommand{\FPSh}[1]{\mathbf{Coh}^{\mathrm{par}}_{\mathrm{flat}}(#1, D)}

\newcommand{\FlKM}[1]{A\text{-}\mathbf{mod}^{\mathrm{fil}}_{\mathrm{flat}}(#1)}
\newcommand{\FPKM}[1]{A\text{-}\mathbf{mod}^{\mathrm{par}}_{\mathrm{flat}}(#1, D)}
\newcommand{\family}[1]{\mathsf{#1}}
\newcommand{\KAl}{\mathbf{A}_{\ell+1} \times \mathbf{K}}
\newcommand{\dv}{\mathbf{d}}
\newcommand{\dimv}[1]{\ensuremath{\mathbf{#1}}}
\newcommand{\mf}{\mathcal{M}}
\newcommand{\msp}{\mathbf{M}}
\newcommand{\mff}{\mathcal{M}^\mathrm{ss}_{\mathrm{fil}}}

\newcommand{\git}{\mathbin{
  \mathchoice{/\mkern-6mu/}
    {/\mkern-6mu/}
    {/\mkern-5mu/}
    {/\mkern-5mu/}}}
\newcommand{\mfs}{\mathscr{M}}
\newcommand{\mffs}{\mathscr{M}^\mathrm{ss}_{\mathrm{fil}}}

\newcommand{\rs}{\mathcal{R}}
\newcommand{\rsc}{\mathcal{R}_{\mathrm{rel}}}
\newcommand{\rsf}{\mathcal{R}_{\mathrm{fil}}}

\newcommand{\msf}{M^\mathrm{ss}_{\mathrm{fil}}}

\newcommand{\mstf}{M^\mathrm{s}_{\mathrm{fil}}}
\newcommand{\fp}[1]{\underline{#1}}
\hypersetup{colorlinks, citecolor=blue, filecolor=black, linkcolor=red}

\begin{document}
\baselineskip=15.5pt

\title{Moduli of Parabolic sheaves and filtered Kronecker modules}
 \author{Sanjay~Amrutiya}
 \address{Department of Mathematics, IIT Gandhinagar,
 Near Village Palaj, Gandhinagar - 382355, India}
 \email{samrutiya@iitgn.ac.in}
 \author{Umesh~V.~Dubey}
 \address{Harish-Chandra Research Institute, HBNI,
Chhatnag Road, Jhunsi, Allahabad - 211019, India}	
 \email{umeshdubey@hri.res.in}
\date{}

\begin{abstract}
\noindent 
We give functorial moduli construction of pure parabolic sheaves, in the sense of \'{A}lvarez-C\'{o}nsul 
and A. King, using the moduli of filtered Kronecker modules which we introduced in our earlier work. We also 
use a version of S. G. Langton's result due to K. Yokogawa to deduce the projectivity of moduli of pure parabolic  sheaves of maximal dimension. As an application of functorial moduli construction, we can get the morphisms at the level of moduli 
stacks.
\end{abstract}
\maketitle
\section{Introduction}
C. S. Seshadri first introduced the notion of parabolic structure on vector bundles over a compact Riemann 
surface, and their moduli constructed by V. B. Mehta and C. S. Seshadri \cite{MS80}. The notion of parabolic 
bundles and several other related concepts and techniques have been generalized from curves to 
higher dimensional varieties by M. Maruyama and K. Yokogawa \cite{MY92}, and later for pure parabolic sheaves 
by M. Inaba \cite{In00}. There have been other moduli constructions given by U. Bhosle \cite{Bh92, Bh96} 
generalizing the notion of the parabolic vector bundles. 

In \cite{MY92, In00}, the authors consider the moduli of stable parabolic sheaves, which they constructed as 
an inductive limit of quasi-projective schemes due to lack of strong boundedness result. In the thesis \cite{Sc11},
D. Schl\"{u}ter  built the moduli construction for pure parabolic sheaves (with some modification in the parabolic 
Hilbert polynomial) by following the method of C. Simpson closely. 

The Biswas-Seshadri correspondence \cite{Bi97} relates parabolic vector bundles to orbifold bundles up to some 
choice of Kawamata cover, and this correspondence was later extended by N. Borne and A. Vistoli \cite{BV12}. 
In \cite{AD15}, we have given functorial moduli construction in the sense of \cite{AK07} of $\Gamma$-sheaves 
on higher dimensional varieties, for a finite group $\Gamma$, by introducing Kronecker McKay modules. 

In this article, we avoid Biswas-Seshadri correspondence and use the filtered object description of parabolic 
sheaves to give the construction of moduli similar to \'{A}lvarez-C\'{o}nsul and A. King \cite{AK07}. We introduce 
the notion of filtered Kronecker module to provide the functorial moduli construction of pure parabolic sheaves. 
We use the structure of moduli of filtered representations done in our earlier work \cite{AD20} for this purpose. 
The moduli construction of filtered representations was done by adopting the moduli  construction of A. King 
in the quasi-abelian category setup with the help of some results of Y. Andr\'e \cite{An09}.
We will now briefly describe the contents of each Section.

We recall the basic notions of parabolic sheaves and its 
reinterpretations in terms of filtered sheaves in Section \ref{prelim} following
\cite{MY92, In00}. We also recall the parabolic Hilbert polynomial which is 
used to give the Gieseker type notion of the stability condition on parabolic sheaves. 

We introduce the notion of parabolic filtered Kronecker module and describe the adjoint functors 
which give the embedding of regular parabolic sheaves in the category of parabolic filtered 
Kronecker modules. We also extend this embedding to the case of flat families 
of these objects in Section \ref{sec-embedding}.

The notion of stability on filtered Kronecker modules is introduced in Section
\ref{sec-ss-analysis} so that the embedding given in the Section \ref{sec-embedding}
takes the parabolic semistable sheaves to the semistable filtered Kronecker 
modules, once we fix the parabolic type.

In Section \ref{sec-construction}, we consider the notion of moduli functors for 
both parabolic sheaves of fixed type and moduli functor of filtered Kronecker 
module. We follow the method of \'{A}lvarez-C\'{o}nsul and King \cite{AK07} to 
give the functorial moduli construction in this Section and describe the closed 
points of moduli space (see Theorem \ref{main_thm:FMC}). More precisely, the
closed point of the coarse moduli scheme are in bijection with the $S$-equivalence 
classes of semistable parabolic sheaves of given parabolic type. This answer the 
question raised in \cite[p. 82]{Sc11} that how strictly semistable orbits in the 
GIT quotients are identified. We also prove the projectivity, using S. G. Langton's result, of moduli space of 
pure parabolic sheaves of fixed type $\tau_p$ such that $\mathrm{deg}(P) = \dim X$.

We extend the functorial relation between moduli functors at the level of moduli 
stacks of parabolic sheaves and filtered Kronecker modules in the last Section 
\ref{sec-stack-embedding}. 
\section{Preliminaries}\label{prelim}
Let $X$ be a projective scheme over an algebraically closed field $\Bbbk$ of arbitrary characteristic, 
let $\struct{X}(1)$ a very ample invertible sheaf on $X$ and $D$ an effective divisor on $X$. By a sheaf on 
$X$, we shall mean a coherent $\struct{X}$-module. 

\begin{definition}\label{MY-defn}\cite{MY92, In00}\rm{
Let $\sheaf{E}$ be a pure sheaf of dimension $d$ such that 
$\dim (D\cap \mathrm{Supp}(\sheaf{E})) < \dim \mathrm{Supp}(\sheaf{E})$.
A \emph{quasi-parabolic structure} on $\sheaf{E}$ with respect to $D$ is a filtration
$$
\sheaf{E} = F_1(\sheaf{E})\supset F_2(\sheaf{E})\supset \cdots \supset F_\ell(\sheaf{E})\supset 
F_{\ell+1}(\sheaf{E}) = \sheaf{E}(-D)\,,
$$
where $l$ is called the length of the filtration. We denote a quasi-parabolic structure on $\sheaf{E}$ by 
$(\sheaf{E}, F_\bullet(\sheaf{E}))$ and this pair is referred as a quasi-parabolic sheaf. 

A parabolic sheaf on $X$ is a quasi-parabolic sheaf $(\sheaf{E}, F_\bullet(\sheaf{E}))$ together with a  sequence of real 
numbers (called weights)
$0\leq \alpha_1 < \alpha_2 < \cdots < \alpha_\ell <  1$. 
}
\end{definition}

We shall often denote the parabolic sheaf $(\sheaf{E}, F_\bullet(\sheaf{E}), \alpha_\bullet)$ by 
$(\psheaf{E},  \alpha_\bullet)$ or simply by $\psheaf{E}$, when it causes no confusion.

Let $(\psheaf{F},  \beta_\bullet)$ and $(\psheaf{E},  \alpha_\bullet)$ be two parabolic sheaves. We say that an $\struct{X}$-module 
homomorphism $f\colon \sheaf{F}\ra \sheaf{E}$ is a parabolic homomorphism if 
$f(F_i(\sheaf{F}))\subseteq F_{j + 1}(\sheaf{E})$ whenever $\beta_i > \alpha_j$ (cf. \cite{MS80}). We denote by $\PSh$ the 
category of parabolic sheaves on $X$ with parabolic structure on $D$.

A parabolic sheaf $(\psheaf{F},  \beta_\bullet)$ is called a parabolic subsheaf of $(\psheaf{E},  \alpha_\bullet)$, if $\sheaf{F} \subseteq \sheaf{E}$
 and $F_i(\sheaf{F}) \subseteq F_{j + 1}(\sheaf{E})$  whenever $\beta_i > \alpha_j$.

Let $(\psheaf{E},  \alpha_\bullet)$ be a parabolic sheaf on $X$ and $\sheaf{F}$ be a non-zero subsheaf of $\sheaf{E}$ such 
that the quotient $\sheaf{E}/\sheaf{F}$ is pure of dimension $d$. Put $F_i(\sheaf{F}) := F_i(\sheaf{E}) \cap \sheaf{F}$. 
After discarding the equalities and by choosing the maximal index among equalities from the following filtration
$$
\sheaf{F} = F_1(\sheaf{F})\supseteq F_2(\sheaf{F})\supseteq \cdots \supseteq F_\ell(\sheaf{F})\supseteq 
F_{\ell+1}(\sheaf{F}) = \sheaf{F}(-D)\,,
$$
we get the (strict) filtration on $\sheaf{F}$. The resulting filtration along with the corresponding weights, say $\tilde{\alpha}_{\bullet}$,
will give the parabolic structure on the subsheaf $\sheaf{F}$, and we call the parabolic sheaf $(\psheaf{F},  \tilde{\alpha}_{\bullet})$ 
the induced parabolic subsheaf of $(\psheaf{E},  \alpha_\bullet)$.

Let $(\psheaf{E},  \alpha_\bullet)$ be a parabolic sheaf. For a real number $\alpha$, take an integer $i$ such that 
$\alpha_{i-1} < \alpha - \lfloor \alpha \rfloor \leq \alpha_i$, where $\lfloor \alpha \rfloor$ is the largest integer 
with $\alpha - \lfloor \alpha \rfloor \geq 0$, and then put 
$\sheaf{E}_\alpha = F_i(\sheaf{E})(- \lfloor \alpha \rfloor D)$. This way one obtain an 
$\R$-filtration $\sheaf{E}_*$ of sheaves on $X$ satisfying the following:
\begin{enumerate}
\item[(i)] $\sheaf{E}_0 = \sheaf{E}$,
\item[(ii)] For all $\alpha < \beta$, $\sheaf{E}_\beta$ is a subsheaf of $\sheaf{E}_\alpha$,
\item[(iii)] For sufficiently small $\epsilon$, $\sheaf{E}_{\alpha - \epsilon} = \sheaf{E}_\alpha$ for all $\alpha$,
\item[(iv)] For all $\alpha$, $\sheaf{E}_{\alpha+1} = \sheaf{E}_\alpha(-D)$.
\end{enumerate}

Conversely, given an $\R$-filtration $\sheaf{E}_*$ of sheaves on $X$ satisfying the above properties, 
then the sheaf $\sheaf{E} := \sheaf{E}_0$ has a unique parabolic structure giving the $\R$-filtration 
$\sheaf{E}_*$ \cite{MY92}.

\begin{remark}\cite{Yo95, Bo07}\rm{
Let us regard $\R$ as an index category, whose objects are real numbers and a unique morphism 
$\beta \ra \alpha$ exists, by definition, precisely when $\beta \leq \alpha$. Consider a functor 
$E_*\colon \R^\mathrm{op}\ra \Coh$, where $\Coh$ is the category of coherent $\struct{X}$-modules. 
If $\alpha \in \R$, then we simply write $E_\alpha$ for $E_*(\alpha)$ and $i_{\alpha, \beta}$ for the morphism 
$E_\alpha \ra E_\beta$ given by the functor $E_*$ when $\alpha \geq \beta$. Given $E_*$ as above and 
$\gamma \in \R$, one can define a new functor $E[\gamma]_* \colon \R^\mathrm{op}\ra \Coh$ as follows:
$$
E[\gamma]_\alpha := E_{\alpha + \gamma},
$$
together with obvious definition on morphisms. If $\gamma \geq 0$, then there is a natural transformation 
$E[\gamma]_* \ra E_*$. The functor $E_*$ is called a \emph{$\R$-parabolic sheaf} if it comes with the following:
\begin{enumerate}
\item each $E_\alpha$ is pure satisfying 
$\dim (D\cap \mathrm{Supp}(E_\alpha)) < \dim \mathrm{Supp}(E_\alpha)$,
\item there is an isomorphism of functors 
$j\colon E_*\otimes \struct{X}(-D)\ra E[1]_*$ such that the diagram
\[
\xymatrix{
E_*\otimes \struct{X}(-D)\ar[r]^-j \ar[rd]_{k_{X, E_*}} & E[1]_* \ar[d] \\
& E_*
}
\]
commutes, where $k_{X, E_*} = \id_{E_*} \otimes i_D$.
\item there is a finite sequence of real numbers $0\leq \alpha_1 < \alpha_2 < \cdots < \alpha_\ell <  1$
such that if $\alpha \in (\alpha_{i-1}, \alpha_i]$, then $i_{\alpha_i, \alpha}\colon E_{\alpha_i}\ra E_\alpha$
is the identity map.
\end{enumerate} 

A homomorphism $f_* : F_* \ra E_*$ between $\R$-parabolic sheaves is nothing but a natural transformation between these two functors.
This is also equivalent to $\struct{X}$-module 
homomorphism $f\colon F_0 \ra E_0$ such that $f(F_\alpha)\subseteq E_\alpha$ for all real numbers $\alpha$. 
We shall denote by $\RPSh$ the category of $\R$-parabolic sheaves on $X$.

It is clear that Maruyama-Yokogawa defintions \cite{MY92} of parabolic sheaf on $X$ (Definition \ref{MY-defn}) 
is equivalent to the notion of $\R$-parabolic sheaf.

}
\end{remark}

\subsection{The parabolic Hilbert polynomial and semistability} 
The parabolic Euler characteristic of a parabolic sheaf $\psheaf{E}$ is defined as
$$
\pchi{\psheaf{E}} := \chi(\sheaf{E}(-D)) + \sum_{i = 1}^\ell \alpha_i \chi(G_i),
$$
where $G_i := F_i(\sheaf{E})/F_{i+1}(\sheaf{E})$.

The parabolic Hilbert polynomial of $\psheaf{E}$ is the polynomial with rational coefficients given by
$$
\pH{\psheaf{E}}(m) := \pchi{\psheaf{E}(m)},
$$
where $\psheaf{E}(m) := \psheaf{E}\otimes \struct{X}(m)$.

Recall that if $\sheaf{E}$ has Hilbert polynomial $P$ and $\sheaf{E}/F_i(\sheaf{E})$ has Hilbert 
polynomial $P_i$ for $2\leq i\leq \ell+1$, then we have
\[
\begin{array}{lll}
\pH{\psheaf{E}}(m) & = P_{\sheaf{E}(-D)}(m) + \sum_{i=1}^\ell \alpha_i P_{G_i}(m)\\
&\\
& = \alpha_{\ell+1}P_{F_{\ell+1}(\sheaf{E})}(m) + \sum_{i=1}^\ell \alpha_i [(P_{F_i(\sheaf{E})}(m) - P_{F_{i+1}(\sheaf{E})}(m)]\\
&\\
& = \alpha_1 P_{F_1(\sheaf{E})}(m) + \sum_{i=2}^{\ell+1} (\alpha_i - \alpha_{i-1}) P_{F_i(\sheaf{E})}(m)\\
&\\
& = \alpha_{\ell+1} P_{\sheaf{E}}(m) + \sum_{i=2}^{\ell+1} (\alpha_i - \alpha_{i-1}) [P_{F_i(\sheaf{E})}(m) - P_{\sheaf{E}}(m)]\\
&\\
& = P_{\sheaf{E}}(m) - \sum_{i=2}^{\ell+1} \varepsilon_i [P_{\sheaf{E}}(m) - P_{F_i(\sheaf{E})}(m)]\\
&\\
& = P(m) - \sum_{i=2}^{\ell+1} \varepsilon_i P_i(m)
\end{array}
\]
where $\varepsilon_i = \alpha_i - \alpha_{i-1}$ and $\alpha_{\ell+1}:= 1$.

We also note that
\[
\begin{array}{lll}
\pchi{\psheaf{E}} & = \chi(\sheaf{E}(-D)) + \sum_{i = 1}^\ell \alpha_i [\chi(F_i(\sheaf{E})) - \chi(F_{i+1}(\sheaf{E}))]\\
&\\
& = \alpha_{\ell+1} \chi(\sheaf{E}(-D)) + \sum_{i = 1}^\ell \alpha_i [\chi(F_i(\sheaf{E})) - \chi(F_{i+1}(\sheaf{E}))]\\
&\\
& = \alpha_1\chi(F_1(\sheaf{E})) + \sum_{i = 2}^{\ell+1} (\alpha_i - \alpha_{i-1}) \chi(F_i(\sheaf{E}))\\
&\\
& = \displaystyle \int_0^1 \chi(\sheaf{E}_\alpha)d\alpha
\end{array}
\]

The reduced parabolic Hilbert polynomial $\rphilb{\psheaf{E}}$ of a parabolic sheaf $\psheaf{E}$ of 
dimension $d$ is defined by
$$
\rphilb{\psheaf{E}}(m) := \frac{\pH{\psheaf{E}}(m)}{a_d(\sheaf{E})}\,,
$$
where $a_d(\sheaf{E})$ is the leading coefficient of the Hilbert polynomial of $\sheaf{E}$.

Recall that the parabolic degree
$$
\pdeg{\psheaf{E}} = a_{d-1}(\sheaf{E}(-D)) + \sum_{i=1}^\ell \alpha_i a_{d-1}(G_i)\,.
$$

The parabolic slope of $\psheaf{E}$ is defined as 
$$
\pslope{\psheaf{E}}:= \frac{\pdeg{\psheaf{E}}}{a_d(\sheaf{E})}\,.
$$
It is easy to see that
$$
\pslope{\psheaf{E}} = \int_0^1 \mu(\sheaf{E}_\alpha) d\alpha\,,
$$
where $\mu(\sheaf{E}_\alpha) = \frac{a_{d-1}(\sheaf{E}_\alpha)}{a_d(\sheaf{E}_\alpha)}$.

\begin{definition}\rm{
We say that a parabolic sheaf $\psheaf{E}$ is parabolic semistable  if for every parabolic subsheaf 
$\psheaf{F}$ of $\psheaf{E}$, we have
\begin{equation}\label{eq:pss-ineq}
\rphilb{\psheaf{F}} \leq \rphilb{\psheaf{E}}\,. 
\end{equation}
We say that a parabolic sheaf $\psheaf{E}$ is parabolic stable  if the inequality \ref{eq:pss-ineq} is strict 
for every proper parabolic subsheaf $\psheaf{F}$ of $\psheaf{E}$.
}
\end{definition}

\begin{remark}\cite[p. 121]{In00}\rm{
Recall that a subsheaf $\sheaf{F}$ of a pure sheaf $\sheaf{E}$ of dimension $d$ is called saturated if 
the quotient sheaf $\sheaf{E}/\sheaf{F}$ is pure of dimension $d$. To check the stability of $\psheaf{E}$, 
it suffices to consider the saturated subsheaves of a parabolic sheaf $\psheaf{E}$ with their induced 
parabolic structures. To see this, let $\psheaf{F}'$ be any parabolic subsheaf of $\psheaf{E}$ and 
$\sheaf{F}$ be the subsheaf of $\sheaf{E}$ containing $\sheaf{F}'$ such that $\dim(\sheaf{F}/\sheaf{F}') < d$ 
and $\sheaf{E}/\sheaf{F}$ is pure of dimension $d$. Let $\tilde{\psheaf{F}}$ be the induced parabolic 
subsheaf of $\psheaf{E}$. Then for sufficiently large integer $m$, we have
\[
\begin{array}{lcl}
\rphilb{\psheaf{F}'}(m) & = \frac{1}{a_d(\sheaf{F}')}\displaystyle \int_0^1 \chi(\sheaf{F}'_\alpha(m))d\alpha \\
& \leq \frac{1}{a_d(\sheaf{F})} \displaystyle \int_0^1 \chi(\sheaf{F}_\alpha(m))d\alpha
& = \rphilb{\sheaf{F}}(m)
\end{array}
\]
}
\end{remark}
\section{A functorial embedding}\label{sec-embedding}
In this section, we first recall an embbeding of the category of regular sheaves
into the category of representations of a Kronecker quiver \cite{AK07}. Then, 
we will extend such result to parabolic case.

For integers $m > n$, let $T := \struct{X}(-n)\oplus \struct{X}(-m)$, and 
\[
A := \begin{pmatrix} \Bbbk & H \\
\empty \\
 0 & \Bbbk \end{pmatrix}
\]
be the path algebra of the Kronecker quiver $K: 1 \stackrel{H}{\longrightarrow} 2$, 
where $H := H^0(\struct{X}(m-n))$ is the multiplicity space for arrows.
Recall that a representation $M$ of $K$ can be described as the decomposition
$M = M_1 \oplus M_2$ together with a linear map 
$\alpha\colon M_1\otimes_{\Bbbk} H\ra M_2$.
Then, $M$ can also be considered as an $A$-module, which will be referred as \emph{Kronecker module}.
For any coherent sheaf $\sheaf{E}$, we have 
$\Hom[X]{T}{\sheaf{E}} = H^0(\sheaf{E}(n))\oplus H^0(\sheaf{E}(m))$ together with
the multiplication map 
$\alpha_{\sheaf{E}}\colon H^0(\sheaf{E}(n))\otimes H \ra H^0(\sheaf{E}(m))$.
Thus, we obtained a functor 
$$
\Phi := \Hom[X]{T}{-} \colon \Coh \ra \KM
$$
given by $\sheaf{E} \mapsto \Hom[X]{T}{\sheaf{E}}$. The functor $\Phi$ has a left adjoint
$$
\Phi^\vee := -\otimes_A T : \KM \ra \Coh
$$
Let $\varepsilon \colon \Phi^\vee \circ \Phi \ra \id_\Coh$ and 
$\eta \colon \id_{\KM} \ra \Phi \circ \Phi^\vee$ be the \emph{co-unit} and 
\emph{unit} of the adjunction between $\Phi$ and $\Phi^\vee$, respectively. 
Let $\nCoh$ be the full subcategory of $\Coh$ consisting of $n$-regular sheaves 
on $X$. Let $\nKM$ be the full subcategory of $\KM$ consisting of $A$-modules $M$ for 
which $\eta_M$ is an isomorphism and $\Phi^\vee(M)$ is $n$-regular.

\begin{proposition}\cite[Theorem 3.4]{AK07}\label{prop-usual-embed}
If $\struct{X}(m-n)$ is regular, then we have an embedding
$$
\Phi \colon \nCoh \ra \KM \,.
$$
Further, $\Phi$ gives an equivalence between  $\nCoh$ and $\nKM$.
\end{proposition}

Consider the linear quiver $\mathbf{A}_{\ell+1}$:
$$
1 \longleftarrow 2 \longleftarrow \cdots \longleftarrow \ell \longleftarrow \ell+1
$$
Let $\CKM$ (respectively, $\chCoh$) denote the category whose objects are representations of a 
linear quiver $\mathbf{A}_{\ell+1}$ in the category of Kronecker modules (respectively, coherent sheaves on $X$).
More precisely, an object $M_\bullet$ (respectively, $E_\bullet$) of $\CKM$ (respectively, $\chCoh$) 
consists of a family of Kronecker modules $M_i$ (respectively, coherent sheaves $E_i$) indexed by the 
vertices of $\mathbf{A}_{\ell+1}$ together with a family of morphisms $M_{i+1}\ra M_i$ (respectively, $E_{i+1}\ra E_i)$
indexed by the arrows in $\mathbf{A}_{\ell+1}$. The morphisms are defined in the usual sense.

Let $\FKM$ be the full subcategory of $\CKM$ consisting of those representations of $\mathbf{A}_{\ell+1}$ 
for which each morphism $M_{i+1}\ra M_i$ is injective. The objects of the category $\FKM$ are called  
\emph{filtered Kronecker modules}.

\begin{remark}\rm{
We get a functor
$
\Phi_{\mathrm{ch}} \colon \chCoh \ra \CKM
$
defined by $\Phi_{\mathrm{ch}}(E_\bullet):= \Phi (E_\bullet)$.
The functor $\Phi_{\mathrm{ch}}$ has a left adjoint 
$$
\Phi^\vee_{\mathrm{ch}} \colon \CKM \ra \chCoh\,.
$$
We also denote by $\varepsilon$ and $\eta$ the \emph{co-unit} and \emph{unit} 
of the adjunction between $\Phi_{\mathrm{ch}}$ and $\Phi^\vee_{\mathrm{ch}}$, respectively.
}
\end{remark}

\begin{definition}\label{PFKM-defn}\rm{
A quasi-parabolic Kronecker module $M_\bullet$ of length $\ell$ is a filtration 
$$
M_\bullet : M = M_1 \supset M_2 \supset \cdots \supset M_\ell \supset 
M_{\ell+1} = M(-D)
$$
in $\nKM$, where $M(-D):= \Phi(\Phi^\vee(M)(-D))$.

A parabolic filtered Kronecker module is a pair $(M_\bullet, \alpha_\bullet)$ consisting of a 
quasi-parabolic filtered Kronecker module $M_\bullet$ and a sequence of real numbers 
$\alpha_\bullet: 0\leq \alpha_1 < \alpha_2 < \cdots < \alpha_\ell <  1$ (called weights associated to a filtration). 
We will denote  the parabolic filtered Kronecker module $(M_{\bullet}, \alpha_\bullet)$ simply by $M_\bullet$,
when it causes no confusion.
}
\end{definition}

Note that any homomorphism $f: N \to M$ of Kronecker modules in $\nKM$ induces a homomorphism 
$f(-D) :=  \Phi(\Phi^\vee(f)(-D))\colon N(-D)\ra M(-D)$, using the functoriality. A homomorphism 
$f_{\bullet} : (N_{\bullet}, \beta_\bullet) \to (M_{\bullet}, \alpha_\bullet)$ of parabolic filtered Kronecker 
modules is defined as a homomorphism 
$f: N \to M$ of Kronecker modules satisfying $f(N_i) \subset M_{j + 1}$ whenever $\beta_i > \alpha_j$. 
We denote by $\PKM$ the category of all parabolic filtered Kronecker modules.

\begin{remark}\rm{
If weights in the parabolic filtered Kronecker modules $N_{\bullet}$ and $M_{\bullet}$ are same, then a 
homomorphism $f_{\bullet} : N_{\bullet} \to M_{\bullet}$ is same as a homomorphism $f: N \to M$ which 
preserves the quasi-parabolic structures, i.e. $f (N_i) \subset M_i$, for $2 \leq i \leq \ell+1$. In particular, 
there is a faithful functor from the category of parabolic Kronecker modules of length $\ell$ having fixed 
weights to the category $\CKM$.
}
\end{remark}

A parabolic sheaf $\psheaf{E}$ is called $n$-regular if each $F_i(\sheaf{E})$ is Castelnuovo-Mumford 
$n$-regular in the usual sense for all $i = 1, 2, \dots, \ell + 1$. We denote by $\PSh^{\mathrm{n\text{-}reg}}$
the category of $n$-regular parabolic sheaves on $X$ with the parabolic structure over $D$.

\emph{Now onwards in this section, we will assume that $\struct{X}(m-n)$ is regular.}

Let $\psheaf{E}$ be an $n$-regular parabolic sheaf. Then, we have a filtration
\begin{equation}\label{eq-par-fil}
\sheaf{E} = F_1(\sheaf{E})\supset F_2(\sheaf{E})\supset \cdots \supset F_\ell(\sheaf{E})\supset 
F_{\ell+1}(\sheaf{E}) = \sheaf{E}(-D)\,,
\end{equation}
together with the weights $0\leq \alpha_1 < \alpha_2 < \cdots < \alpha_\ell <  1$. 

Since $\struct{X}(m-n)$ is regular, we have $\Phi^\vee\Phi(\sheaf{E}) \cong \sheaf{E}$, and hence
$$\Phi(\sheaf{E})(-D) := \Phi(\Phi^\vee(\Phi(\sheaf{E}))(-D)) = \Phi(\sheaf{E}(-D))\,.$$
We can define a quasi-parabolic structure on $\Phi(\psheaf{E})$ as follows by applying the functor $\Phi$ to
\eqref{eq-par-fil}:
\begin{equation}\label{eq-Phi-fil}
\Phi(\sheaf{E}) = \Phi(F_1(\sheaf{E}))\supset \Phi(F_2(\sheaf{E}))\supset \cdots \supset \Phi(F_\ell(\sheaf{E}))\supset 
\Phi(F_{\ell+1}(\sheaf{E})) = \Phi(\sheaf{E}(-D))\,,
\end{equation}
where $\Phi(\sheaf{E})(-D) = \Phi(\sheaf{E}(-D))$. By assigning the same weights $\alpha_\bullet$ to the filtration \eqref{eq-Phi-fil}, 
we get a structure of parabolic filtered Kronecker module on $\Phi(\sheaf{E})$, which we denote by $\Psi(\psheaf{E})$. 
Now, if $f\colon \psheaf{F}\ra \psheaf{E}$ is a morphism of parabolic sheaves, then 
$\Psi(f) = \Phi(f) \colon \Phi(\sheaf{F})\ra \Phi(\sheaf{E})$ is morphism of corresponding parabolic 
Kronecker modules. To see this, if $\beta_i > \alpha_j$, then $f$ being a morphism of parabolic sheaves,
we have $f(F_i(\sheaf{F})) \subseteq F_{j+1}(\sheaf{E})$. Since $\Phi$ preserves monomorphism, we have
$\Psi(f)(\Phi(F_i(\sheaf{F})) \subseteq \Phi(F_{j+1}(\sheaf{E}))$. Therefore, we get a functor
$\Psi \colon \PSh^{\mathrm{n\text{-}reg}} \ra \PKM$. 

\begin{proposition}\label{prop-psh-emb}
The functor $\Psi \colon \PSh^{\mathrm{n\text{-}reg}} \ra \PKM$ is an embedding of categories.
\end{proposition}
\begin{proof}
Let $\mathcal{C}$ be a full subcategory of $\PKM$ consisting of those parabolic filtered Kronecker modules
$M_\bullet$ for which we have a filtration
\begin{equation}\label{Psi-chek}
\Phi^\vee(M) \supset \Phi^\vee(M_2) \supset \cdots \Phi^\vee(M_\ell) \supset \Phi^\vee(M(-D))\,.
\end{equation}
Note that for an object $M$ of $\nKM$, we have 
$$
\Phi^\vee(M(-D)) = \Phi^\vee(\Phi((\Phi^\vee M)(-D))) = (\Phi^\vee M)(-D)
$$
For an object $M_\bullet$ of $\mathcal{C}$, by assigning the same weights as in $M_\bullet$ to the filtration 
\eqref{Psi-chek}, we get a parabolic structure on $\Phi^\vee(M)$. Let $\Psi^\vee(M_\bullet)$ denote the 
parabolic sheaf corresponding to an object $M_\bullet$ of $\mathcal{C}$. By definition of $\PKM$, the
parabolic sheaf $\Psi^\vee(M_\bullet)$ is $n$-regular. Hence, we get a functor
$\Psi^\vee \colon \mathcal{C}\ra \PSh^{\mathrm{n\text{-}reg}}$.

Let $\sheaf{E}_\bullet$ be an $n$-regular parabolic sheaf. Using Proposition \ref{prop-usual-embed},
we can check that $\Psi(\sheaf{E}_\bullet)$ is an object of $\mathcal{C}$ and 
$\Psi^\vee (\Psi(\sheaf{E}_\bullet)) \cong \sheaf{E}_\bullet$.
In other words, the functor $\Psi$ factors through the 
category $\mathcal{C}$. By the above construction, we have a left adjoint functor 
$\Psi^{\vee}: \mathcal{C} \ra \PSh^\mathrm{n\text{-}reg}$. Since both unit and co-unit of the adjunction are 
isomorphisms, we get that the functor $\Psi$ factors through an equivalence to a subcategory of $\PKM$.
\end{proof}

\begin{remark}\rm{
We can define the parabolic filtered Kronecker module using $\R$-filtration following the notion of 
$\R$-parabolic sheaves from Section \ref{prelim} and the functorial emebdding. We shall denote this 
category by $\RPKM$ whose objects and morphisms are described as follows:
  
\begin{description}
\item[Objects]  All functors $M_*\colon \R^\mathrm{op}\ra \nKM$ satisfying the 
following:
\begin{enumerate}
\item[(a)] there is an isomorphism of functors 
$j_{A, M_*} \colon M_*(-D)\ra M[1]_*$ such that the diagram
\[
\xymatrix{
M_*(-D)\ar[r]^-{j_{A, M_*}} \ar[rd]_{k_{A, M_*}} & M[1]_* \ar[d]^{i_{M_*}^{[0, 1]}} \\
& M_*
}
\]
commutes, where $M_*(-D) := \Phi(\Phi^\vee(M_*)(-D))$ and 
$$k_{A, M_*} = \eta_{M_*} \circ \Phi(k_{X, \Phi^\vee(M_*)}) \colon M_*(-D)\ra M_* \,.$$
\item[(b)] there is a finite sequence of real numbers $0\leq \alpha_1 < \alpha_2 < \cdots < \alpha_\ell <  1$
such that if $\alpha \in (\alpha_{i-1}, \alpha_i]$, then $i_{\alpha_i, \alpha}\colon M_{\alpha_i}\ra M_\alpha$
is the identity map.
\end{enumerate}

\item[Morphisms] A morphism between two objects $(N_*, j_{A, N_*})$ and $(M_*, j_{A, M_*})$ is a functor
$h_*\colon N_*\ra M_*$ making the following diagram
\[
\xymatrix{
**[l] N_*(-D) \ar[r]^{j_{A, N_*}} \ar[d]_{h_*(-D)} & **[r] N_*[1] \ar[d]^{h_*[1]} \\
**[l] M_*(-D) \ar[r]_{j_{A, M_*}} & **[r] M_*[1]
}
\]
commutative, where $h_*(-D) := \Phi(\Phi^\vee(h_*)(-D))$.
\end{description}
The objects of $\RPKM$ are called $\R$-parabolic filtered Kronecker modules.
Similar to the case of parabolic sheaves, we can check that the category $\RPKM$  is equivalent to the 
category of parabolic filtered Kronecker modules.
}
\end{remark}

\subsection{Flat families and stratification}

Let $S$ be a scheme. We say that $(\sheaf{E}, F_\bullet(\sheaf{E}),\alpha_\bullet)$ is a flat family over $S$
of parabolic sheaves on $X$ with the parabolic structure over $D$ if $\sheaf{E}$ is a $S$-flat coherent 
$\struct{X_S}$-module such that for every geometric point $s$ of $S$, $\sheaf{E}_s$ is of pure dimension 
$d$, $\dim (D_s \cap \mathrm{Supp} (\sheaf{E}_s)) < \dim \mathrm{Supp} (\sheaf{E}_s)$ 
and $\sheaf{E} = F_1(\sheaf{E}) \supset \cdots \supset F_{\ell + 1}(\sheaf{E}) = \sheaf{E}(-D)$ is a filtration by 
coherent sheaves such that each $\sheaf{E}/F_i(\sheaf{E})$ is flat over $S$.
In other words, we say that the corresponding $\R$-parabolic sheaf $E_*$ is flat if for each 
$\alpha \in \R$, the sheaf $E_\alpha$ is flat over $S$.
We denote by $\FPSh{X_S}^{\mathrm{n\text{-}reg}}$ the category of flat family over $S$ of $n$-regular 
parabolic sheaves on $X$ with the parabolic structure over $D$. 

Let $\mathbf{K}: 1\stackrel{H}{\longrightarrow} 2$ be a Kronecker quiver, 
where $H$ is the multiplicity of arrows. Consider the ladder quiver 
$\KAl$ whose set of vertices and arrows can be described as follows:
\begin{equation}\label{relation_ideal}
\xymatrix{
1_1 \ar[d]_{H_1} & 2_1 \ar[l]_{\beta_1^1} \ar[d]_{H_2} & \cdots 
 \ar[l]_{\beta_2^1} & \ell_1 \ar[d]_{H_\ell} \ar[l]_{\beta_{\ell-1}^1} & (\ell+1)_1  
 \ar[d]_{H_{(\ell+1)}} \ar[l]_{\beta_\ell^1} \\
1_2 & 2_2 \ar[l]^{\beta_1^2} & \cdots \ar[l]^{\beta_2^2} & \ell_2 
 \ar[l]^{\beta_{\ell-1}^2} & (\ell+1)_2 \ar[l]^{\beta_\ell^2} 
  }
\end{equation}

Let $I$ be the ideal in the path algebra $\Bbbk(\KAl)$ generated by $h_i\beta^1_i - \beta^2_ih_{i+1}$, 
where $h_i\in H_i = H,\; i=1, 2, \dots, \ell$. Set $B= \Bbbk(\KAl)/I$. 

We say that a sheaf $\family{M}_\bullet$ of right modules over the sheaf of algebras $\struct{S}\otimes B$
is flat over $S$, if it is locally-free as sheaf of $\struct{S}$-modules. Moreover, we say that
a flat family $\family{M}_\bullet$ over $S$ is a flat family of filtered Kronecker modules, if
for each closed point $s\in S$, the fibre $\family{M}_{s_\bullet}$ is a filtered Kronecker module.
We denote by $\FlKM{S}$ the category of flat families of filtered Kronecker modules 
over $S$. Furthermore, we say that it is a flat family of parabolic filtered Kronecker modules, if the fibres 
$\family{M}_{s_\bullet}$ are parabolic filtered Kronecker modules.
We denote by $\FPKM{S}$ the category of flat families of parabolic filtered Kronecker modules 
over $S$.

Using \cite[Proposition 4.1]{AK07} and Proposition \ref{prop-psh-emb}, we have the following:

\begin{proposition}\label{s2-prop-eqiv-familiy}
The functor $\Psi\colon \FPSh{X_S}^\mathrm{n\text{-}reg}\ra \FPKM{S}$ is an embedding of categories.
\end{proposition}
\begin{proof}
Since $\Psi$ preserves flat families and monomorphism at each fibre, it follows that the functor $\Psi$ is 
well defined.
We denote by $\mathcal{C}_S$ the full subcategory of $\FPKM{S}$ consisting of flat families of parabolic 
filtered Kronecker modules over $S$ such that the family $\Psi^\vee(\family{M}_\bullet)$ is a flat family 
of $n$-regular parabolic sheaves.

Let $\sheaf{E}_\bullet$ be a flat family over $S$ of $n$-regular parabolic sheaves on $X$. Using 
Proposition \ref{prop-psh-emb}, we can check that $\Psi(\sheaf{E}_\bullet)$ is an object of $\mathcal{C}_S$.
In other words, the functor $\Psi$ factors through the category $\mathcal{C}_S$. By following the construction 
of the left adjoint functor in the proof of Proposition \ref{prop-psh-emb}, the functor 
$\Psi\colon \FPSh{X_S}^\mathrm{n\text{-}reg}\ra \mathcal{C}_S$ has a left adjoint functor 
$\Psi^{\vee}: \mathcal{C}_S \ra \FPSh{X_S}^\mathrm{n\text{-}reg}$. Since both 
unit and co-unit of the adjunction are isomorphisms, we get that the functor $\Psi$ factors through an 
equivalence to a subcategory of $\FPKM{S}$.
\end{proof}

Let $\tau_p$ represent fixed tuple of numerical polynomials $P, P_1, \dots, P_\ell$ 
such that $\deg P = d$ and $\deg P_i < d$ for each $i = 1, 2, \dots, \ell$, and fixed sequence of 
real numbers $0 < \alpha_1 < \alpha_2 < \cdots < \alpha_\ell < 1$.
We say that a parabolic sheaf $\psheaf{E} = (\sheaf{E}, F_\bullet(\sheaf{E}), \alpha_\bullet)$ is of parabolic 
type $\tau_p$ if the following holds:
\begin{itemize}
\item $\chi(\sheaf{E})(k) = P(k), \; \chi(\sheaf{E}/F_{i+1}(\sheaf{E}))(k) = P_i(k)\; 
(i = 1, \dots, \ell)$ for all integers $k$. 
\end{itemize}

\begin{proposition}\label{s2-prop-stratification}
Assume that $\struct{X}(m-n)$ is regular. Let $\family{M}_\bullet$ be a flat family of filtered 
Kronecker modules of dimension vector $\dv$ over a scheme $B$. Then, there exist a unique locally
closed subscheme $\iota\colon B^\mathrm{reg}_{\tau_p} \ra B$ such that the following hold.
\begin{enumerate}
\item $\Psi^\vee(\iota^*\family{M}_\bullet)$ is a flat family over $B^\mathrm{reg}_{\tau_p}$
of $n$-regular parabolic sheaves on $X$ with parabolic type $\tau_p$ and the unit map
$$
\eta_{\iota^*\family{M}_\bullet} \colon \iota^*\family{M}_\bullet\ra 
\Psi\circ \Psi^\vee(\iota^*\family{M}_\bullet)
$$
is an isomorphism.
\item If $\sigma \colon S\ra B$ is such that $\sigma^*\family{M}_\bullet \cong \Psi(E_*)$
for a flat family $E_*$ over $S$ of $n$-regular parabolic sheaves on $X$ with parabolic type 
$\tau_p$, then $\sigma$ factors through $\iota\colon B^\mathrm{reg}_{\tau_p} \ra B$ and
$E_* \cong \Psi^\vee(\sigma^*\family{M}_\bullet)$.
\end{enumerate}
\end{proposition}
\begin{proof}

Using the flattening stratification for $\Psi^\vee \family{M}_\bullet$ over the projection $X\times B\ra B$, 
we can conclude that, there exists a locally closed subscheme
 $j \colon B_{\tau_P} \ra B$ such that $j^*\Psi^\vee \family{M}_\bullet$ is a flat family over $B_{\tau_P}$ of 
 chain of coherent sheaves on $X$, and the closed points of $B_{\tau_P}$ are precisely those $b\in B$ 
 for which the fibers $\family{M}_{b_\bullet}$ have following properties:
 \begin{itemize}
 \item $\Psi^\vee \family{M}_{b_{\ell+1}} = \Psi^\vee \family{M}_{b_1}(-D)$.
 \item for each $i$, the natural map $\Psi^\vee \family{M}_{b_{i+1}} \ra \Psi^\vee \family{M}_{b_i}$ is injective.
 \item the Hilbert polynomial of $\Psi^\vee \family{M}_{b_1}$ is $P$ and the Hilbert polynomial of 
 $\Psi^\vee \family{M}_{b_1}/\Psi^\vee \family{M}_{b_{i+1}}$ is $P_i$ for each $i=1, 2, \dots, \ell$. 
 \end{itemize}
 It is clear that $B_{\tau_P}$ contains an open set $B_{\tau_P}^\mathrm{reg}$ of points $b$ for which the sheaf 
 $\Psi^\vee \family{M}_{b_i}$ is $n$-regular and the morphism 
 $\eta_{\family{M}_{b_i}}\colon \family{M}_{b_i}\ra \Psi(\Psi^\vee (\family{M}_{b_i}))$ is an isomorphism for each 
 $i$ (as both are open conditions). Now, using the proof of Proposition  \ref{s2-prop-eqiv-familiy},
 it follows that $\Psi(\Psi^\vee \family{M}_\bullet)$ is a flat family of parabolic filtered Kronecker modules over 
 $B_{\tau_P}^\mathrm{reg}$ and the unit map $\eta_{\iota^*\family{M}_\bullet} \colon \iota^*\family{M}_\bullet\ra 
\Psi\circ \Psi^\vee(\iota^*\family{M}_\bullet)$ is an isomorphism, where $\iota\colon B^\mathrm{reg}_{\tau_p} \ra B$.
Therefore, we get a locally closed subscheme $\iota\colon B^\mathrm{reg}_{\tau_p} \ra B$
satisfying (1). The proof of (2) follows the similar line of arguments as in \cite[Proposition 4.2]{AK07}.
\end{proof}

\section{Preservation of semi-stability}\label{sec-ss-analysis}
In this section, we will now study the semistability of parabolic sheaves and 
determined the stability parameter for the filtered Kronecker modules. This will 
be useful to reduce the problem of constructing the moduli space to GIT quotient
by choosing an appropriate parameter space into the representation space of filtered
Kronecker modules.

Recall that $\tau_p$ represent fixed tuple of numerical polynomials $P, P_1, \dots, P_\ell$ 
such that $\deg P = d$ and $\deg P_i < d$ for any $i$, and fixed sequence of 
real numbers $\alpha_\bullet = (\alpha_1, \dots, \alpha_\ell)$ 
such that $0 < \alpha_1 < \alpha_2 < \cdots < \alpha_\ell < 1$.

In \cite{In00}, the collection of $e$-stable parabolic sheaves of given type $\tau_p$
is shown to be bounded following the boundedness result of \cite{MY92}. In \cite{Sc11},
using the results of \cite{Si94, La04}, 
the author has extended the boundedness result to the semistable parabolic sheaves 
of given type $\tau_p$.

\begin{theorem}\cite[Theorem 4.5.2]{Sc11}\label{boundedness}
Let $\mathfrak{F}_X^\mathrm{ss}(\tau_p)$ be the family of parabolic semistable 
sheaves on $X$ of parabolic type $\tau_p$. Then, 
the family $\mathfrak{F}_X^\mathrm{ss}(\tau_p)$ is bounded.
\end{theorem}

The above result is very crucial to give a characterization of semistability of parabolic sheaves of a given 
type as in the following Proposition \ref{MY-ss-propn}. This, in turn, helps to determine the right stability 
parameter for the filtered Kronecker modules to give a functorial construction of the coarse moduli space 
of semistable parabolic sheaves having a parabolic type $\tau_p$. 

\begin{proposition}\label{MY-ss-propn}
There exists an integer $N$ such that for any pure $d$-dimensional parabolic sheaf $\psheaf{E}$
on $X$ of parabolic type $\tau_p$, the following are equivalent:
\begin{enumerate}
\item $\psheaf{E}$ is parabolic semistable.
\item For all $n\geq N$, $\psheaf{E}$ is $n$-regular, and for all proper subsheaf $\sheaf{E'}\subset \sheaf{E}$ 
the inequality of 
polynomials
\begin{equation}\label{eq-integral-1}
\Big(\sum_{i=1}^{\ell+1} \varepsilon_i h^0(F_i(\sheaf{E'})(n))\Big)a_d(\sheaf{E}) \leq 
\Big(\sum_{i=1}^{\ell+1} \varepsilon_i h^0(F_i(\sheaf{E})(n))\Big)a_d(\sheaf{E'})
\end{equation}
holds, where $\varepsilon_i = \alpha_i - \alpha_{i-1},\; \alpha_0 := 0,\; \alpha_{\ell+1}:=1$
\end{enumerate}

Moreover, if $\psheaf{E}$ is parabolic semistable, and $\sheaf{E'}$ is a proper subsheaf of $\sheaf{E}$,
then equality holds in \eqref{eq-integral-1} if and only if $\rphilb{\psheaf{E}} = \rphilb{\psheaf{E'}}$.
\end{proposition}
\begin{proof}
If $\psheaf{E}$ is parabolic semistable, then by \cite[Proposition 2.5]{MY92} together with Theorem 
\ref{boundedness}, for any proper subsheaf $\sheaf{E'}$ of $\sheaf{E}$, we have
$$
\frac{1}{a_d(\sheaf{E'})}\int_0^1 h^0(\sheaf{E'}_\alpha(n))d\alpha \leq 
\frac{1}{a_d(\sheaf{E})}\int_0^1 h^0(\sheaf{E}_\alpha(n))d\alpha
$$
for sufficiently large $n$. This is equivalent to \eqref{eq-integral-1}.

Conversely, to check that $\psheaf{E}$ is parabolic semistable, it is enough to show that for any 
saturated subsheaf $\sheaf{E'}$ of $\sheaf{E}$, we have $\rphilb{\psheaf{E'}} \leq \rphilb{\psheaf{E}}$. 
Let $\sheaf{E'}$ be a saturated subsheaf of $\sheaf{E}$. If $\pslope{\psheaf{E'}} < \pslope{\psheaf{E}}$, 
then it can not destabilize the parabolic sheaf $\psheaf{E}$. So, assume that 
$\pslope{\psheaf{E'}} \geq \pslope{\psheaf{E}}$. Then for some $\alpha\in \R$, we have 
$\mu(\sheaf{E}'_\alpha) \geq \pslope{\psheaf{E}}$. Using \cite[Lemma 1.7.9]{HL97}, 
we can conclude that the family of saturated subsheaves of $\sheaf{E}$ with 
the property that $\mu(\sheaf{E}'_\alpha) \geq \mu:=\pslope{\psheaf{E}}$ for some $\alpha$, is bounded. 
In particular, we can conclude that $\psheaf{E'}$ is $n$-regular for sufficiently large $n$. 
Since the inequality \eqref{eq-integral-1} holds for $\sheaf{E'}$, we have
$$
\rphilb{\psheaf{E'}}(n) = 
\frac{1}{a_d(\sheaf{E'})}\Big(\sum_{i=1}^{\ell+1} \varepsilon_i h^0(F_i(\sheaf{E'})(n))\Big) 
\leq 
\frac{1}{a_d(\sheaf{E})}\Big(\sum_{i=1}^{\ell+1} \varepsilon_i h^0(F_i(\sheaf{E})(n))\Big) 
= \rphilb{\psheaf{E}}(n)
$$
for sufficiently large enough $n$. This implies that $\rphilb{\psheaf{E'}} \leq \rphilb{\psheaf{E}}$.
\end{proof}

\begin{proposition}\label{LC-estimate}
There exists an integer $P_{LS}$ such that, for all $n \geq P_{LS}$ the following are equivalent 
for any pure $d$-dimensional parabolic sheaf $\psheaf{E}$ of parabolic type $\tau_p$:
\begin{enumerate}
\item $\psheaf{E}$ is parabolic semistable.
\item For all $n\geq P_{LS}$, $\psheaf{E}$ is $n$-regular, and for all proper subsheaf 
$\sheaf{E'}\subset \sheaf{E}$ the inequality of polynomials
\begin{equation}\label{eq-poly-1}
\sum_{i=1}^{\ell+1} \varepsilon_i h^0(F_i(\sheaf{E'})(n))P_{\sheaf{E}} \leq \pH{\psheaf{E}}(n)P_{\sheaf{E'}}
\end{equation}
holds.
\end{enumerate}

Moreover, if $\psheaf{E}$ is parabolic semistable, and $\sheaf{E'}$ is a proper subsheaf of $\sheaf{E}$, 
then equality holds in \eqref{eq-poly-1} if and only if $\rphilb{\psheaf{E}} = \rphilb{\psheaf{E'}}$.
\end{proposition}
\begin{proof}
This immediately follows from the Proposition \ref{MY-ss-propn} by noting that $a_d(\sheaf{E'})$ and 
$a_d(\sheaf{E})$ are leading terms in the Hilbert polynomials $P_{\sheaf{E'}}$ and $P_{\sheaf{E}}$,
respectively.
\end{proof}

\begin{proposition}\cite[Proposition 4.4.3]{Sc11}\label{ss-prop-2}
For a fixed parabolic type $\tau_p$, the full subcategory of category of all parabolic sheaves on $X$
consisting of semistable parabolic sheaves having parabolic type $\tau_p$ is an abelian, Noetherian 
and Artinian category. Moreover, its simple objects are precisely the stable parabolic sheaves 
having parabolic type $\tau_p$.
\end{proposition}

From Proposition \ref{ss-prop-2}, it follows that a parabolic semistable sheaf $\psheaf{E}$ having 
parabolic type $\tau_p$ admits a Jordan-H\"older filtration
$$
\{0\} = \sheaf{E}_{0_*} \subset \sheaf{E}_{1_*} \subset \sheaf{E}_{2_*} \subset \cdots \subset 
\sheaf{E}_{k_*} = \psheaf{E}
$$
such that the successive quotients $\sheaf{E}_{*_i+1}/\sheaf{E}_{*_i}$ are stable parabolic sheaves
having parabolic type $\tau_p$. The associated graded object is 
$$
\mathrm{gr}(\psheaf{E}) = \oplus_{i=0}^{k-1} \sheaf{E}_{*_i+1}/\sheaf{E}_{*_i}
$$
which does not depend on the choice of the filtration upto an isomorphism. 

We say that two parabolic semistable sheaves $\psheaf{E}$ and $\psheaf{F}$ are $S$-equivalent if
the associated graded objects $\mathrm{gr}(\psheaf{E})$ and $\mathrm{gr}(\psheaf{F})$ are
isomorphic.

Now onwards, in this section we choose $m \gg n \gg 0$ such that the following holds:
\begin{enumerate}
\item[(C1)] All parabolic semistable sheaves of parabolic type $\tau_p$ are $n$-regular.
\item[(C2)] The Le Potier-Simpson estimates hold, i.e., $n \geq P_{LS}$, for $P_{LS}$ as in
Proposition \ref{LC-estimate}.
\item[(C3)] $\struct{X}(m-n)$ is regular.
\end{enumerate}

Let $\psheaf{E}$ be any parabolic sheaf of parabolic type $\tau_p$ which is $n$-regular. For each $j$,
let 
$$
ev_j \colon H^0(F_j(\sheaf{E})(n))\otimes \struct{X}(-n)\ra \sheaf{E}
$$
be the natural evaluation map.
For subspaces $V_j' \subseteq H^0(F_j(\sheaf{E})(n))$, let $E'_j$ and $F'_j$ be the image and 
kernel of restriction of $ev_j$ to $V_j'\otimes \struct{X}(-n)$. Let $\mathcal{S}$ be the
set of all sheaves $E'_j,\; F'_j$ that arises in this way, and all saturated subsheaves 
$\sheaf{E}'\subset \sheaf{E}$, where $\psheaf{E}$ is $n$-regular parabolic sheaf having parabolic type 
$\tau_p$ and $\pslope{\psheaf{E'}} \geq \pslope{\psheaf{E}}$. Then by Grothendieck lemma, the family 
$\mathcal{S}$ is bounded.

\begin{enumerate}
\item[(C4)] All the sheaves in $\mathcal{S}$ are $m$-regular. 
\item[(C5)] For a parabolic sheaf $\psheaf{E}$ of parabolic type $\tau_p$, let $P_j(k) = \chi(F_j(\sheaf{E})(k))$.
Then for any integers $c_j \in \{0, 1, \dots, P_j(n)\}$ and sheaves $E' \in \mathcal{S}$, 
the polynomial relation $P_\sheaf{E}\sum_{j}\varepsilon_j c_j \sim P_{\sheaf{E}'}\pH{\psheaf{E}}(n)$ is 
equivalent to the relation $P_\sheaf{E}(m)\sum_{j}\varepsilon_j c_j \sim P_{\sheaf{E}'}(m)\pH{\psheaf{E}}$, 
where $\sim$ is any of $\leq$ or $<$ or $=$.
\end{enumerate}

\subsection{Semistability of filtered Kronecker modules}
In \cite{AD20}, we gave the moduli construction of filtered quiver representation with the appropriate notion 
of semistability using rank and degree functions. In the following, we will consider a specific rank and degree, 
which corresponds to the stability of parabolic sheaves. 

Let 
$$M_\bullet : M = M_1 \supset M_2 \supset \cdots \supset M_\ell \supset M_{\ell+1}$$ be a
filtered Kronecker module, where $M_i$'s are Kronecker modules. The category of 
all filtered Kronecker modules will be denoted as $\FKM$. We define a slope 
of $M_\bullet$ as
\begin{equation}\label{ss-slope}
\mu(M_\bullet) := \frac{\sum_{i=1}^{\ell+1} \varepsilon_i \dim M_{i1}}{\dim M_{12}},
\end{equation}
which takes values in $[0, \infty]$.

Let $\dv = (d_{11}, d_{12}, \dots, d_{(\ell+1) 1}, d_{(\ell+1) 2})$ be a dimension 
vector of $\KAl$. If $M_\bullet$ is an object of $\FKM$ having dimension vector 
$\dv$, we set $\mathrm{rk}(\dv) = d_{12}$. 

Let $M_\bullet'$ be a non-zero subrepresentation of $M_\bullet$ in the category 
of filtered Kronecker modules. We say that $M_\bullet'$ is \emph{degenerate} 
if $M_{12}' = 0$.

Let $\Theta_{j1} = \varepsilon_j$ and $\Theta_{j2} = 0$. Then, we get degree function
$$
\Theta(\dv(M_\bullet)):= \sum_{w\in (\KAl)_0} \Theta_w \dv(M)_w \,.
$$

If $M_\bullet$ is an object of $\FKM$ having dimension vector 
$\dv$ and $M_\bullet'$ is a subrepresentation of $M_\bullet$,
we set
\begin{equation}\label{eq-theta-defn}
\theta(M_\bullet') := \big(\sum_{j=1}^{\ell+1} (\Theta_{j1} \dim M_{j1}')\big)d_{12} - 
\big(\sum_{j=1}^{\ell+1} (\Theta_{j1} d_{j1})\big) \dim M'_{12}
\end{equation}

where $\Theta_{j1} = \varepsilon_j$ and $\Theta_{j2} = 0$ for each $j$.

We say that $M_\bullet$ is $\theta$-semistable if $\theta(M_\bullet') \leq 0$ 
for all subrepresentations $M'_\bullet$ of $M_\bullet$.
Note that the slope $\mu(M'_\bullet)$ is well-defined for all non-degenerate 
subrepresentations of $M_\bullet$. 

\begin{lemma}\label{ss-lemma-3}
Let $\psheaf{E}$ be an $n$-regular parabolic sheaf having parabolic type $\tau_p$, and let 
$M_\bullet := \Psi(\psheaf{E})$ be a corresponding filtered Kronecker module. 
Then, $M_\bullet$ does not have any degenerate subrepresentation.
\end{lemma}
\begin{proof}
This follows from the observation in the proof of \cite[Lemma 8.8]{GRT}.  
\end{proof}

We can reformulate the $\theta$-semistability in terms of slope semistability.

\begin{proposition}\label{ss-prop-1}
Assume that a filtered Kronecker module $M_\bullet$ having dimension vector $\dv$
does not have any degenerate subrepresentation. Then, $M_\bullet$ is 
$\theta$-semistable if and only if for all non-zero subrepresentations 
$M'_\bullet$ of $M_\bullet$, we have $\mu(M'_\bullet)\leq \mu(M_\bullet)$.
\end{proposition}
\begin{proof}
Note that
$$
\theta(M_\bullet') := \big(\sum_{j=1}^{\ell+1} (\varepsilon_j \dim M_{j1}')\big)d_{12} - 
\big(\sum_{j=1}^{\ell+1} (\varepsilon_j d_{j1})\big) \dim M'_{12}
$$
Hence, we have
$$
\theta(M'_\bullet) = d_{12}\dim M'_{12}(\mu(M'_\bullet) - \mu(M_\bullet)).
$$
From this the assertion follows.
\end{proof}

\begin{definition}\rm{
Let $M'_\bullet$ and $M''_\bullet$ be two subobjects of $M_\bullet$. We say that 
$M'_\bullet$ is subordinate to $M''_\bullet$ if 
$$
M'_{j1} \subseteq M''_{j1} ~\mbox{for 
each}~ j, ~\mbox{and}~ M''_{12} \subseteq M'_{12}.
$$

We say that $M'_\bullet$ is tight if whenever it is subordinate to a subobject $M''_\bullet$ of $M_\bullet$, we have
$M'_{j1} = M''_{j1}$ for each $j$, and $M''_{12} = M'_{12}$. In this case, we have $\mu(M'_\bullet) = \mu(M''_\bullet)$.
}
\end{definition}

\begin{lemma}\label{ss-lemma-1}
Let $\widetilde{M}_\bullet$ be a subobject of $M_\bullet$. Then $\widetilde{M}_\bullet$ 
is subordinate to some tight subobject $M'_\bullet$ of $M_\bullet$.
\end{lemma}
\begin{proof}
Let $M_\bullet$ be a filtered Kronecker module. In other words, we have linear maps
$$
\rho_j \colon M_{j1}\otimes H\ra M_{j2}
$$
together with injective linear maps
$$
f_{j1}\colon M_{(j+1)1}\ra M_{j1} \quad \mbox{and} \quad f_{j2}\colon M_{(j+1)2}\ra M_{j2}
$$
such that $f_{j2} \circ \rho_{j+1} = \rho_j \circ f_{j1}$ for each $j = 1, 2, \dots , \ell$. 

Let $M'_{12} := \rho_1(\widetilde{M}_{11}\otimes H)$ and
$M'_{11} := \{v\in M_{11}\;|\; \alpha_j(v\otimes h) \in M'_{12} \; \mbox{for all}\; h\in H\}$. For $j\geq 2$, we define
$M'_{j1} := (f_{11}\circ f_{21} \circ \cdots \circ f_{(j-1)1})^{-1}(M'_{11})$ and $M'_{j2} := \rho_j(M'_{j1}\otimes H)$. Since
each $f_{j1}$ is injective, we have $M'_{j1} = M'_{11}\cap M_{j1}$ for each $j$.

Let $M'_\bullet$ be a subobject of $M_\bullet$ given by $\{M'_{j1}, M'_{j2}\}$ with the induced maps
from $\rho'$s and $f'$s. From the definition, it follows that $\widetilde{M}_{11}\subseteq M'_{11}$ and 
$\widetilde{M}_{12}\supseteq M'_{12}$. Since $f_{j1}(\widetilde{M}_{(j+1)1})\subseteq \widetilde{M}_{j1}$ and
$\widetilde{M}_{11}\subseteq M'_{11}$, by induction, we can conclude that $\widetilde{M}_{j1}\subseteq M'_{j1}$ for all
$j\geq 2$. This proves that $\widetilde{M}_\bullet$ is subordinate to $M'_\bullet$.

To see that $M'_\bullet$ is a tight subobject of $M_\bullet$, let $M''_\bullet$ be a subobject of $M_\bullet$ such that
$M'_\bullet$ is subordinate to $M''_\bullet$. That is, we have $M'_{j1}\subseteq M''_{j1}$ and $M''_{12}\subseteq M'_{12}$.
Note that
$$
M'_{12} = \rho_1(\widetilde{M}_{11}\otimes H) \subseteq \rho_1(M'_{11}\otimes H) 
\subseteq \rho_1(M''_{11}\otimes H) \subseteq M''_{12} 
$$
This proves that $M''_{12} = M'_{12}$, and hence by definition of $M'_{11}$, we have $M'_{11} = M''_{11}$. 
Observe that $M''_{j1} \subseteq M''_{11} = M'_{11}$, and hence $M''_{j1} \subseteq (M'_{11}\cap M_{j1}) = M'_{j1}$ for $j\geq 2$.
This completes the proof that $M'_\bullet$ is a tight subobject of $M_\bullet$.
\end{proof}

\begin{lemma}\label{ss-lemma-2}
Let $\psheaf{E}$ be an $n$-regular parabolic sheaf with parabolic type $\tau_p$. Let 
$M_\bullet = \Psi(\psheaf{E})$, and let $M'_\bullet$ be a subobject of $M_\bullet$. 
Then there exists a subsheaf $\sheaf{E}'$ of $\sheaf{E}$ such that $M'_\bullet$ is subordinate 
to $\Psi(\psheaf{E}')$. In particular, if $M'_\bullet$ is tight and non-degenerate, then 
$\mu(M'_\bullet) = \mu(\Psi(\psheaf{E}'))$.
\end{lemma}
\begin{proof}
Let $\psheaf{E}: E_1 \supset E_2 \supset \cdots E_\ell \supset E(-D)$, and let
$M_\bullet : M_1 \supset M_2 \supset \cdots M_\ell \supset M(-D)$. Let $M'_\bullet$ be a subobject of 
$M_\bullet$. Let
$E'_j := \mathrm{Im}(M'_{j1} \otimes \struct{X}(-n)\ra E)$, and $\sheaf{E}' := E'_1$.
Then, the $\sheaf{E}'$ being a subsheaf of $\sheaf{E}$, we get an induced parabolic structure on $\sheaf{E}'$,
which we shall denote by $\psheaf{E}'$. Note that
$$
M'_{j1} \subset H^0(E'_j(n)) \subset H^0(F_j(\sheaf{E}')(n))\,.
$$
For each $j$, there is a short exact sequence $0 \ra F_j \ra M'_{j1}\otimes \struct{X}(-n)\ra E'_j\ra 0$.
By (C4), $F_j$ is $m$-regular, and hence the map $M'_{j1}\otimes H^0(\struct{X}(m-n))\ra H^0(E'_j(m))$ is
surjective.  Since $M'_\bullet$ is a subobject of $M_\bullet$, it follows that 
$$
H^0(F_1(\sheaf{E}')(m)) =  H^0(\sheaf{E}'(m))\subset M'_{12}\,.
$$
This proves that $M'_\bullet$ is subordinate to $\Psi(\psheaf{E}')$.
\end{proof}

\begin{remark}\rm{
Let $\mathrm{sk}(\FKM)$ be the skeleton of $\FKM$, that is, the set of isomorphism 
classes of objects of $\FKM$.
From the above Lemma \ref{ss-lemma-3}, we get a rank function
$$
\mathrm{rk}\colon \mathrm{sk}(\FKM) \ra \mathbb{N}
$$
given by $\mathrm{rk}(\dv) = d_{12}$. 

We can consider a slope function $\mu \colon \mathrm{sk}(\FKM)-\{0\}\ra \mathbb{Q}$ 
as follows:
$$
\mu_{\Theta, \mathrm{rk}}(M):= \frac{\Theta(\dv(M))}{\mathrm{rk}(\dv(M))}\,.
$$
Note that the slope defined in \eqref{ss-slope} agree with $\mu_{\Theta, \mathrm{rk}}$
on $\FKM$.
}
\end{remark}

\begin{theorem}\label{ss-thm-1}
A parabolic sheaf $\psheaf{E}$ of parabolic type $\tau_p$ is parabolic semistable 
if and only if it is pure, $n$-regular
and $\Psi(\psheaf{E})$ is $\theta$-semistable.
\end{theorem}
\begin{proof}
Suppose that $\psheaf{E}$ is parabolic semistable. Then, by definition, it is pure 
and $n$-regular by (C1). By Lemma \ref{ss-lemma-3} and Proposition \ref{ss-prop-1},
we need to check that $\mu(M'_\bullet)\leq \mu(M_\bullet)$ for all tight subobjects 
$M'_\bullet$ of $M_\bullet$.
Let $M'_\bullet$ be a tight subobject of $M_\bullet$. By Lemma \ref{ss-lemma-2}, 
we get a parabolic subsheaf $\psheaf{E}'$ of $\psheaf{E}$ such that 
$\mu(M'_\bullet) = \mu(\Psi(\psheaf{E}'))$. Since $\psheaf{E}$ is parabolic 
semistable, by Proposition \ref{LC-estimate}, we have
$$
\sum_{i=1}^{\ell+1} \varepsilon_i h^0(F_i(\sheaf{E'})(n))P_{\sheaf{E}} \leq 
\pH{\psheaf{E}}(n)P_{\sheaf{E'}}\,.
$$
This polynomial inequality implies the following numerical inequality:
$$
\sum_{i=1}^{\ell+1} \varepsilon_i h^0(F_i(\sheaf{E'})(n))P_{\sheaf{E}}(m) \leq 
\pH{\psheaf{E}}(n)P_{\sheaf{E'}}(m)\,.
$$
In other words,
$$
\mu(M'_\bullet) = 
\frac{\sum_{i=1}^{\ell+1} \varepsilon_i h^0(F_i(\sheaf{E'})(n))}{P_{\sheaf{E'}}(m)} 
\leq
\frac{\sum_{i=1}^{\ell+1} \varepsilon_i h^0(F_i(\sheaf{E})(n))}{P_{\sheaf{E}}(m)} 
= \mu(M_\bullet).
$$ 
This proves that $\Psi(\psheaf{E})$ is $\theta$-semistable.

Conversely, suppose that $\psheaf{E}$ is pure, $n$-regular and $\Psi(\psheaf{E})$ 
is $\theta$-semistable. Let $\sheaf{E}'$ be a saturated subsheaf of $\sheaf{E}$ 
such that $\pslope{\psheaf{E'}} \geq \pslope{\psheaf{E}}$.
Since $\Psi(\psheaf{E})$ is $\theta$-semistable, we have 
$\theta(\Psi(\psheaf{E}')) \leq 0$.
Note that
$$
\theta(\Psi(\psheaf{E}')) = 
\big(\sum_{i=1}^{\ell+1} \varepsilon_i h^0(F_i(\sheaf{E}')(n))\big) P_{\sheaf{E}}(m)
-  \big(\sum_{i=1}^{\ell+1} \varepsilon_i h^0(F_i(\sheaf{E})(n))\big) P_{\sheaf{E}'}(m)
$$
Hence, we have
$$
\big(\sum_{i=1}^{\ell+1} \varepsilon_i h^0(F_i(\sheaf{E}')(n))\big) P_{\sheaf{E}}(m)
\leq  \big(\sum_{i=1}^{\ell+1} \varepsilon_i h^0(F_i(\sheaf{E})(n))\big) P_{\sheaf{E}'}(m)
$$
By condition (C5), this numerical inequality is equivalent to the polynomial inequality
$$
\big(\sum_{i=1}^{\ell+1} \varepsilon_i h^0(F_i(\sheaf{E}')(n))\big) P_{\sheaf{E}}
\leq  \big(\sum_{i=1}^{\ell+1} \varepsilon_i h^0(F_i(\sheaf{E})(n))\big) P_{\sheaf{E}'}
$$
Now, using Proposition \ref{LC-estimate}, we can conclude that $\psheaf{E}$ is 
parabolic semistable.
\end{proof}

From the Theorem \ref{ss-thm-1} and Proposition \ref{prop-psh-emb}, we can deduce that the 
functor $\Psi$ induces an equivalence between the category $\PSh^\mathrm{ss}(\tau_p)$ consisting 
of semistable parabolic sheaves on $X$ having parabolic type $\tau_p$ and the full subcategory category 
$\mathcal{C}^\mathrm{ss}(\mu)$ of $\mathcal{C}$ consisting of semistable parabolic filtered Kronecker 
module having fixed slope $\mu$, which is determined by the parabolic type $\tau_p$. 
Since $\PSh^\mathrm{ss}(\tau_p)$ is an abelian category, it follows that Jordan-H\"older theorem holds
for $\mathcal{C}^\mathrm{ss}(\mu)$. In other words, if $M_\bullet$ is an object of 
$\mathcal{C}^\mathrm{ss}(\mu)$, then there exists a filtration
$$
\{0\} = M_{0_\bullet} \subset M_{1_\bullet} \subset M_{2_\bullet} \subset \cdots \subset 
M_{k_\bullet} = M_\bullet
$$
such that the successive quotients $M_{(i+1)_\bullet}/M_{i_\bullet}$ are stable objects
in $\mathcal{C}^\mathrm{ss}(\mu)$. The associated graded object is 
$$
\mathrm{gr}^{\mathrm{JH}}(M_\bullet) = \oplus_{i=0}^k M_{(i+1)_\bullet}/M_{i_\bullet}
$$
which does not depend on the choice of the filtration upto an isomorphism. 

We say that two objects $M_\bullet$ and $N_\bullet$ in $\mathcal{C}^\mathrm{ss}(\mu)$ are 
$S^{\mathrm{JH}}$-equivalent if the associated graded objects $\mathrm{gr}^\mathrm{JH}(M_\bullet)$ and 
$\mathrm{gr}^{\mathrm{JH}}(N_\bullet)$ are isomorphic.

\begin{proposition}\label{ss-prop-3}
Let $\psheaf{E}$ be a semistable parabolic sheaf having parabolic type $\tau_p$, and let 
$M_\bullet := \Psi(\psheaf{E})$ be the corresponding filtered Kronecker module. 
If $\sheaf{E}_{1_\bullet}\subset \sheaf{E}_{2_\bullet}$ are parabolic subsheaves of $\psheaf{E}$
having parabolic type $\tau_p$,
then
$$
\Psi(\sheaf{E}_{2_\bullet})/\Psi(\sheaf{E}_{1_\bullet})\cong \Psi(\sheaf{E}_{2_\bullet}/\sheaf{E}_{1_\bullet})\,.
$$
Moreover, we have
$$
\mathrm{gr}^{\mathrm{JH}}(M_\bullet) \cong \Psi(\mathrm{gr}(\psheaf{E}))\,. \quad \mbox{and} \quad
\mathrm{gr}(\psheaf{E}) \cong \Psi^\vee(\mathrm{gr}^{\mathrm{JH}}(M_\bullet))\,.
$$
\end{proposition}
\begin{proof}
First note that if $\psheaf{E}'$ is a parabolic subsheaf of $\psheaf{E}$ having 
parabolic type $\tau_p$, then $\psheaf{E}'$ is also semistable, and hence 
$n$-regular by the condition (C1). Consider the short exact sequence
\begin{equation}\label{s4-eq-exact}
0\ra \sheaf{E}_{1_\bullet}\ra \sheaf{E}_{2_\bullet}\ra \sheaf{E}_{2_\bullet}/\sheaf{E}_{1_\bullet}\ra 0\,.
\end{equation}
For each $\alpha$, we have a short exact sequence
$$
0\ra \sheaf{E}_{1_\alpha}\ra \sheaf{E}_{2_\alpha}\ra 
\sheaf{E}_{2_\alpha}/\sheaf{E}_{1_\alpha}\ra 0\,.
$$
Since $\sheaf{E}_{1_\bullet}$ is $n$-regular, by applying the functor $\Psi$ to the 
short exact sequence \eqref{s4-eq-exact}, we get the isomorphism
$$
\Psi(\sheaf{E}_{2_\bullet})/\Psi(\sheaf{E}_{1_\bullet})\cong \Psi(\sheaf{E}_{2_\bullet}/\sheaf{E}_{1_\bullet})\,.
$$
The rest of the proof follows in the same line as in \cite[Corollary 5.11]{AK07}.
\end{proof}
\section{Moduli functors and construction}\label{sec-construction}
We are now ready to give a functorial construction of the following moduli problem.
Consider the moduli functor
$$
\mf(\tau_p)^\mathrm{ss}_X\colon (\Sch/\Bbbk)^\circ \ra \Set
$$
which assigns to a $\Bbbk$-scheme $S$ the set of isomorphism classes of flat families over $S$ of 
semistable parabolic sheaves on $X$ having parabolic type $\tau_p$. Similarly, there is a moduli functor
$\mf(\tau_p)^\mathrm{s}_X$ for stable sheaves.

Let us fix the dimension vector $\dimv{d}$ for the ladder quiver 
$$\KAl = ((\KAl)_0, (\KAl)_1)\,,$$
where $(\KAl)_0$ is the set of vertices and
$(\KAl)_1$ is the set of arrows as described in \eqref{relation_ideal}.  Let 
$$
\rs := \bigoplus_{a\in (\KAl)_1} 
\Hom[\Bbbk]{\Bbbk^{d_{s(a)}}}{\Bbbk^{d_{t(a)}}}
$$ 
be the space of representations of $\KAl$ having dimension vector $\dimv{d}$.
Let $\rsc$ be the closed subset of $\rs$ consisting of $(\alpha_a)$ satisfying the 
relation in $I$ (see, \eqref{relation_ideal}). Let $\rsf$ be an open subset of $\rsc$ consisting of 
$(\alpha_a)\in \rsc$ such that for any $a\in \mathbf A_{\ell 1} \times Q_0$, we have 
$\alpha_a$ injective map. Let
$$
G(\KAl):= \displaystyle \prod_{w\in (\KAl)_0} \mathrm{GL}(\Bbbk^{d_w})\,.
$$
There is a natural action of the group $G(\KAl)$
on $\rs$ such that the isomorphism classes in $\RepKAl$ correspond to the orbits in 
$\rs$ with respect to this action. 
Let $G:= G(\KAl)/\Delta$, where
$\Delta:= \{(t\mathbf{1}_{\Bbbk^{d_w}})_{w\in (\KAl)_0}\,|\, t\in \Bbbk^*\}$ 
and there is an action of this group $G$ on $\rs$.

Note that $\rsf$ is a locally closed subscheme of $\rs$, and is invariant 
under this action. Given a weight $\sigma\in \R^{(\KAl)_0}$, let 
$$
\rsf^{\sigma\text{-}\mathrm{ss}} = 
\{x\in \rsf \;|\; \mbox{the corresponding representation}~M_x~\mbox{is}~
\sigma\text{-}\mbox{semistable}\}
$$ 
be the subset of $\rsf$, which is $\sigma$-semistable locus in $\rsf$. 
Similarly, let $\rsf^{\sigma\text{-}\mathrm{s}}$ denote the $\sigma$-stable locus in $\rsf$.
Recall that if $\sigma$ is an integral weight, then $\rsf^{\sigma\text{-}\mathrm{ss}}$ and
$\rsf^{\sigma\text{-}\mathrm{s}}$ are open subset of $\rsf$. By \cite[Corollary 2.3]{Ch08}, 
it follows that they are open subsets of $\rsf$ for arbitrary weights. 

Now, for a fixed parabolic type $\tau_p$, we have a fixed parabolic weights 
$\alpha_*$. From this, we get a weight $\theta\in \R^{(\KAl)_0}$ as in \eqref{eq-theta-defn}
so that the functor $\Psi$ preserve the semistability (see, Section \ref{sec-ss-analysis}).
More precisely, for a vertex $w \in (\KAl)_0$, let
\[
\theta_w = \left\{\begin{array}{ll}
\varepsilon_j d_{12} & \mbox{if}~ w = j_1 \;, j = 1, 2, \cdots, \ell+1 \\
& \\
-\sum_{i=1}^{\ell+1}\varepsilon_i d_{i1} & \mbox{if}~ w = 1_2\\
& \\
0 & \mbox{otherwise},
\end{array}\right.
\]
where $\dv = (d_{11}, d_{12}, \dots, d_{(\ell+1) 1}, d_{(\ell+1) 2})$ is the dimension
vector of $\KAl$ determined by the fixed parabolic type $\tau_p$.

By \cite[Corollary 2.3]{Ch08}, there exist an integral weight $\sigma\in \Z^{(\KAl)_0}$
such that 
$$
\rsf^{\sigma\text{-}\mathrm{ss}} = \rsf^{\theta\text{-}\mathrm{ss}}~\mbox{and}~
\rsf^{\sigma\text{-}\mathrm{s}} = \rsf^{\theta\text{-}\mathrm{s}}\,.
$$

Let $\chi_\sigma \colon G\ra \Bbbk^*$ be the character determined by the 
integral weight $\sigma$. Then, we have
$$
\rsf^{\chi_\sigma\text{-}\mathrm{ss}} = \rsf^{\theta\text{-}\mathrm{ss}}~\mbox{and}~
\rsf^{\chi_\sigma\text{-}\mathrm{s}} = \rsf^{\theta\text{-}\mathrm{s}}\,.
$$
where $\rsf^{\chi_\sigma\text{-}\mathrm{ss}}$ denote the open subset of $\rsf$
which is $\chi_\sigma$-semistable locus in $\rsf$.

For simplicity, we shall denote $\rsf^{\theta\text{-}\mathrm{ss}}$ and 
$\rsf^{\theta\text{-}\mathrm{s}}$ simply by $\rsf^\mathrm{ss}$ and $\rsf^\mathrm{s}$,
respectively.

Recall \cite[Section 4.2]{AD20}, the moduli functor
$$
\mf^\mathrm{ss}_\mathrm{fil}(\dimv{d}) \colon (\Sch/\mathbb K)^\circ \ra \Set
$$
which assigns to a $\Bbbk$-scheme $S$ the set of isomorphism classes of flat families over $S$ of 
semistable filtered Kronecker modules of given dimension vector $\dimv{d}$.

\begin{remark}\label{remark:well-definedness}\rm{
Let $\mathbb{M}_\bullet$ be a tautological family of filtered Kronecker modules on $\rsf$. It will be a 
trivial vector bundle for each vertex of the ladder quiver, as it is a restriction of the tautological family 
over the affine space $\rs$. We can use this fact to see that if two families are related by an action of 
group element over a base scheme $S$, then they are isomorphic as families. 
}
\end{remark}

There is a natural functor $h\colon \fp{\rsf^\mathrm{ss}}\ra 
\mf^\mathrm{ss}_\mathrm{fil}(\dimv{d})$, which induces a local isomorphism 
$\tilde{h}\colon \fp{\rsf^\mathrm{ss}}/\fp{G}\ra \mf^\mathrm{ss}_\mathrm{fil}(\dimv{d})$ 
\cite[Proposition 4.11]{AD20}. 

\begin{theorem}
There exist moduli spaces $\msf(\dimv{d})$ (respectively, $\mstf(\dimv{d}))$
of $\theta$-semistable (respectively, $\theta$-stable) filtered Kronecker modules
having dimension vector $\dimv{d}$, where $\msf(\dimv{d})$ is a quasi-projective scheme.
Further, the closed points of the moduli space $\msf(\dimv{d})$ (respectively, $\mstf(\dimv{d})$) 
correspond to the $S$-equivalence classes of $\theta$-semistable 
(respectively, $\theta$-stable) objects of $\FKM$ which have dimension vector $\dimv{d}$. 
\end{theorem}
\begin{proof}
In view of Proposition \ref{ss-prop-1}, this is a special case of \cite[Theorem 4.12]{AD20}.
\end{proof}

Let $Q = (\rsf)^{\mathrm{reg}}_{\tau_p}$ be the locally closed subscheme of $\rsf$ corresponding 
to the family $\mathbb{M}_\bullet$ as in the Proposition \ref{s2-prop-stratification}.

Let $\mf_1, \mf_2 \colon \Sch^\mathrm{op}\ra \Set$ be two functors. Recall that a morphism of 
functors $g\colon \mf_1 \ra \mf_2$ is said to be a local isomorphism, if it induces 
an isomorphism of sheafification in the Zariski topology.

\begin{lemma}
For a fixed parabolic type $\tau_p$, the moduli functor $\mf(\tau_p)^\mathrm{ss}_X$ is locally isomorphic 
to the quotient functor $\fp{Q^{\mathrm{[ss]}}}/\fp{G}$.
\end{lemma}
\begin{proof}
Let $Q^{\mathrm{[ss]}}\subset Q$ be an open loci, where the fibres of the tautological flat family 
$\mathbb{F}_\bullet = \Psi^\vee(\iota^* \mathbb{M}_\bullet)$ over $Q$ are semistable parabolic sheaves.
Then, we get a natural transformation
$$
g^\mathrm{ss}\colon \fp{Q^{\mathrm{[ss]}}} \ra \mf(\tau_p)^\mathrm{ss}_X
$$
which is restriction of the natural transformation 
$g\colon \fp{Q} \ra \mf(\tau_p)^\mathrm{reg}_X$ defined by
$$
g^\mathrm{ss}(\sigma\colon S\ra Q^{\mathrm{[ss]}})\mapsto [\Psi^\vee(\sigma^* \iota^* \mathbb{M}_\bullet)]
$$
If we change the family $\sigma^* \iota^* \mathbb{M}_\bullet$ by the action of the group $G(S)$, then as 
we noticed in Remark \ref{remark:well-definedness}, the tautological family comes with 
the trivial vector bundle over each vertex of the ladder quiver, and hence gives the isomorphic family. 
Since an isomorphism of families is preserved after applying the functor $\Psi^\vee$, we get the well defined 
natural transformation $\tilde{g}^\mathrm{ss} \colon \fp{Q^{\mathrm{[ss]}}}/\fp{G} \ra \mf(\tau_p)^\mathrm{reg}_X$.
By Proposition \ref{s2-prop-stratification}(1), we have the following commutative diagram of natural 
transformations
\[
\xymatrix{
\fp{Q^{\mathrm{[ss]}}} \ar[r]^\iota \ar[d]_{g^\mathrm{ss}} & \fp{\rsf^{\mathrm{ss}}} 
\ar[d]^{h} \\
\mf(\tau_p)^\mathrm{ss}_X \ar[r]_j & \mff(\dv) 
}
\]
For any $\Bbbk$-scheme $S$, the map $\fp{Q^{\mathrm{[ss]}}} \ra 
\mf(\tau_p)^\mathrm{ss}_X\times_{\mff(\dv)} \fp{\rsf^{\mathrm{ss}}}$
defined by 
$$
(\sigma\colon S\ra Q^{\mathrm{[ss]}})\mapsto (g_S(\sigma), \iota\circ \sigma)
$$
is a bijection, by Proposition \ref{s2-prop-stratification}(2). In other words, we have the
following Cartesian diagram
\[
\xymatrix{
\fp{Q^{\mathrm{[ss]}}}/\fp{G} \ar[r] \ar[d]_{\tilde{g}^\mathrm{ss}} & 
\fp{\rsf^{\mathrm{ss}}}/\fp{G} \ar[d]^{\tilde{h}} \\
\mf(\tau_p)^\mathrm{ss}_X \ar[r]_j & \mff(\dv) 
}
\]
Since $\tilde{h}$ is a local isomorphism, it follows that 
$\tilde{g}^\mathrm{ss}\colon \fp{Q^{\mathrm{[ss]}}}/\fp{G} \ra \mf(\tau_p)^\mathrm{ss}_X$
is a local isomorphism.
\end{proof}

\begin{theorem}\label{main_thm:FMC}
There exists a quasi-projective scheme $\msp_X^\mathrm{ss}(\tau_p)$ which 
corepresents the moduli functor
$\mf(\tau_p)^\mathrm{ss}_X$. The closed points of $\msp_X^\mathrm{ss}(\tau_p)$ correspond to 
the $S$-equivalence classes of semistable parabolic sheaves having parabolic type $\tau_p$.
\end{theorem}
\begin{proof}
First we show that the good quotient 
$\pi_X \colon Q^{\mathrm{[ss]}}\ra Q^{\mathrm{[ss]}}\sslash G=:\msp_X^\mathrm{ss}(\tau_p)$
exists. Set $Z := \overline{Q^{\mathrm{[ss]}}}\cap \rsf^\mathrm{ss}$. Let $O(M_\bullet)$ be
an orbit of a point in $Q^{\mathrm{[ss]}}$ corresponding to $M_\bullet = \Psi(\psheaf{E})$, 
where $\psheaf{E}$ is a semistable parabolic sheaf on $X$ having parabolic type $\tau_p$.
The closed orbit in $Z$ contained in the closure of the orbit $O(M_\bullet)$ in $Z$ is the
orbit corresponding to the associated graded object $\mathrm{gr}^\mathrm{JH}(M_\bullet)$ 
(see, \cite[Proposition 4.9]{AD20}). By Proposition \ref{ss-prop-3}, 
$\mathrm{gr}^{\mathrm{JH}}(M_\bullet) \cong \Psi(\mathrm{gr}(\psheaf{E}))\,,$ and
$\mathrm{gr}(\psheaf{E})$ is semistable. Hence, this closed orbit is also in $Q^{\mathrm{[ss]}}$.
Now following the arguments as in the proof of \cite[Proposition 6.3]{AK07}, we conclude that
the good quotient 
$\pi_X \colon Q^{\mathrm{[ss]}}\ra \msp_X^\mathrm{ss}(\tau_p)$ exists.

Since $\tilde{g}^\mathrm{ss}\colon \fp{Q^{\mathrm{[ss]}}}/\fp{G} \ra \mf(\tau_p)^\mathrm{ss}_X$
is a local isomorphism, it follows that $\msp_X^\mathrm{ss}(\tau_p)$ corepresents the moduli functor
$\mf(\tau_p)^\mathrm{ss}_X$.

Moreover, we have the following commutative diagram
\[
\xymatrix{
Q^{\mathrm{[ss]}} \ar[r]^\iota \ar[d]_{\pi_X} & \rsf^\mathrm{ss} \ar[d]^\pi \\
\msp_X^\mathrm{ss}(\tau_p) \ar[r]_\varphi & \msf(\dv)
}
\]
where the morphism $\varphi \colon \msp_X^\mathrm{ss}(\tau_p) \ra \msf(\dv)$ is a set-theoretic
injection of the closed points \cite[Proposition 6.3]{AK07}. In other words, it induces a bijection
$$
\varphi\colon \msp_X^\mathrm{ss}(\tau_p)(\Bbbk) \ra \pi(Q^{\mathrm{[ss]}})(\Bbbk)\,.
$$
Hence, the closed points of $\msp_X^\mathrm{ss}$ correspond to the $S$-equivalence classes of
semistable filtered Kronecker modules that are of the form $\Psi(\psheaf{E})$ for semistable parabolic 
sheaves $\psheaf{E}$ having parabolic type $\tau_p$. By Proposition \ref{ss-prop-3}, the assertion for
closed points follows.
\end{proof}

\begin{remark}\rm{
The image of the map 
$\varphi\colon \msp_X^\mathrm{ss}(\tau_p)(\Bbbk) \ra \msf(\dv)(\Bbbk)$ is contained in the set 
of all closed points of $\msf(\dv)$ which correspond to the $S^{JH}$-equivalence classes of filtered 
Kronecker modules (see the remark after \cite[Theorem 4.12]{AD20}). We have not considered the case of 
parabolic weights with $\alpha_1 = 0$, as we do not know the validity of \cite[Proposition 4.4.3]{Sc11}. 
But in these cases, it is possible to use the more general notion of $S$-equivalence as in \cite{AD20} to get 
the description of two types of points of $\msp_X^\mathrm{ss}(\tau_p)$. 
}
\end{remark}

The projectivity of moduli space $\msp_X^\mathrm{ss}(\tau_p)$ follows
esentially from \cite[\S 5]{Yo93} using Langton's criterion. 
In the following, we provide some essential steps with appropriate references.

Let $R$ be a discrete valuation ring over $\Bbbk$ with 
quotient filed $L$. Let $i\colon X_\Bbbk\ra X_R$ be a closed immersion, 
and $j\colon X_L\ra X_R$ an open immersion.

\begin{proposition}\rm{\cite[Theorem 5.7]{Yo93}}\label{s5-prop-extn}
Let $\psheaf{E}$ be a flat family of parabolic sheaves having parabolic 
type $\tau_p$ on $X$ over $\mathrm{Spec}(R)$ such that $j^*\psheaf{E}$ is a 
flat family of semistable parabolic sheaves on $X$ over $\mathrm{Spec}(L)$. 
Then, there exists an $\struct{X_R}$-submodule $\sheaf{E}'$ of $\sheaf{E}$
such that $j^*\sheaf{E}' = j^*\sheaf{E}$ and $\psheaf{E}'$ is a flat family
of parabolic sheaves on $X$ over $\mathrm{Spec}(R)$ and $i^*\psheaf{E}'$ is
semistable; where $\sheaf{E}'_\alpha = \sheaf{E}'\cap \sheaf{E}_\alpha$ for
all $\alpha$.
\end{proposition}

\begin{proposition}
Let $\tau_p$ be a fixed parabolic type with $\deg(P) := \dim X$.
The moduli space $\msp_X^\mathrm{ss}(\tau_p)$ is projective over $\mathrm{Spec}(\Bbbk)$.
\end{proposition}
\begin{proof}
Since $\msp_X^\mathrm{ss}(\tau_p)$ is quasi-projective, we only need to show that it is proper
over $\mathrm{Spec}(\Bbbk)$. For this, let $R$ be a discrete valuation ring with quotient 
field $L$ and residue filed $\Bbbk$. Let $x\colon \mathrm{Spec}(L) \ra X^\mathrm{ss}(\tau_p)$.
Then, there exist a field extension $L'$ of $L$ and $L'$-valued point $x'$ of $Q^\mathrm{[ss]}$
such that the following diagram 
\[
\xymatrix{
\mathrm{Spec}(L') \ar[r]^{x'} \ar[d] & Q^\mathrm{[ss]} \ar[d]^\pi \\
\mathrm{Spec}(L) \ar[r]_{x} & \msp_X^\mathrm{ss}(\tau_p)
}
\]
commutes. Let $\psheaf{E}' = x'^*\mathbb{F}_\bullet$, where $\mathbb{F}_\bullet = 
\Psi^\vee(\iota^* \mathbb{M}_\bullet)$ is the tautological family of parabolic 
semistable sheaves having parabolic type $\tau_P$ on $Q^\mathrm{[ss]}$. 
Let $R'$ be a discrete valuation ring dominating $R$ with quotient filed $L'$. 
Let $\sheaf{F}$ be a coherent $\struct{X_{R'}}$-submodule 
of $j_*\sheaf{E}'$ such that $j^*\sheaf{F} = \sheaf{E}'$, where
$j\colon X_L \ra X_{R'}$ (cf. \cite[Proposition 6]{La75}). 
By \cite[Exercise 2.B.2]{HL97}, there is a 
subsheaf $\sheaf{F}'\subset \sheaf{F}$ such that $j^*\sheaf{F}' = j^*\sheaf{F} 
= \sheaf{E}'$ and $i^*\sheaf{F}'$ is pure, where $i\colon X_{\Bbbk}\ra X_{R'}$ 
is a closed immersion. 
By properness of the Quot
scheme, there exists a unique subsheaf $\sheaf{F}'_\alpha$ of $\sheaf{F}'$ 
such that $j^*\sheaf{F}'_\alpha = \sheaf{E}'_\alpha$ and 
$\sheaf{F}'/\sheaf{F}'_\alpha$ is flat over $R'$ for each $0\leq \alpha \leq 1$. 
Since $i^*\sheaf{F}'$ is pure and hence using the assumption on $\tau_p$, we get 
$\dim(\mathrm{Supp}(i^*(\sheaf{F}')) \cap D) < \dim(\mathrm{Supp}(i^*(\sheaf{F}'))$. Therefore, the map 
$i^*(\sheaf{F}'\otimes_{\struct{X_R'}} \struct{X_R'}(-D_{R'}))\ra i^*\sheaf{F}'$ 
is injective, and hence 
$\sheaf{F}'/(\sheaf{F}'\otimes_{\struct{X_R'}} \struct{X_R'}(-D_{R'}))$ is flat 
over $R'$.

By Proposition \ref{s5-prop-extn}, we get a flat family 
$\psheaf{E}''$ of semistable parabolic sheaves on $X$ over $\mathrm{Spec}(R')$ 
such that $j^*\psheaf{E}'' = \psheaf{E}'$. The corresponding classifying map 
$x'_{R'}\colon \mathrm{Spec}(R')\ra \msp_X^\mathrm{ss}(\tau_p)$
is the lift of $x'$. The rest of the proof follows the identical arguments 
as in the proof of \cite[Proposition 6.6]{AK07}.
\end{proof}

\section{Moduli stacks and embedding}\label{sec-stack-embedding}
Recall that a quasi-parabolic sheaf on a scheme $X$ is called a sheaf of fixed type if each terms of the 
associated graded sheaf has fixed Hilbert polynomial. It is called $n$-regular if each of the sub-sheaves 
in the filtration is $n$-regular. We can check that the quasi-parabolic sheaf is $n$-regular if each 
subquotients are $n$-regular. We will say that the parabolic sheaf is $n$-regular, if the underlying 
quasi-parabolic sheaf is $n$-regular.

\begin{definition} \rm{
The moduli stack of $n$-regular parabolic sheaves is a category fibered in groupoids defined as,
$$P_X : \mfs_X^\mathrm{reg}(\tau_P) \to \Sch/\mathbb K ; \mfs^\mathrm{reg}_X(\tau_P) (T) \mapsto T$$
where $\mfs_X^\mathrm{reg}(\tau_P) (T)$ is the groupoid of all flat families of $n$-regular parabolic sheaves 
of fixed type parametrised by the scheme $T$ and morphism is given by isomorphism of families of 
quasi-parabolic sheaves. The morphisms of the category $\mfs_X^{reg}(\tau_P) $ are pairs 
$(f : T' \to T, \phi)$, similar to the example described in \cite[Example 2.16]{Go01}. 
}
\end{definition}

There are also two open moduli sub-stacks of the moduli stack $ \mfs_X^\mathrm{reg}$ determined by the 
notion of parabolic stability condition. The moduli stack of semistable (respectively, stable) parabolic 
sheaves will be denoted as $ \mfs_X^{\mathrm{ss}}(\tau_P)$ (respectively, $ \mfs_X^{\mathrm{s}}(\tau_P)$).

\begin{remark}\rm{
Using the boundedness result of Schl\"uter \cite[Theorem 4.5.2]{Sc11}, originally due to Maruyama and 
Yokogawa, and its extension by Inaba in some special cases, we can conclude that semistable and stable 
sub-stacks are algebraic stacks with respect to fppf topology on $\Sch$. Moreover, we get a canonical 
maps $Q^{[ss]} \ra \mfs_X^{\mathrm{ss}}(\tau_P)$  and $Q^{[s]} \ra \mfs_X^{\mathrm{s}}(\tau_P)$
}
\end{remark}

We will denote by $\mathbf M_X^{\mathrm{ss}}(\tau_P)$ (respectively, $\mathbf M_X^{\mathrm{s}}(\tau_P)$) 
the coarse moduli space of semistable (respectively, stable) parabolic moduli stack (see Theorem \ref{main_thm:FMC}). 
Similar to the parabolic case, we can also define the moduli stack of representations of fixed dimension vector.
 
\begin{definition}\rm{
The moduli stack of filtered representations is a category fibered in groupoids defined as,
$$ P_A : \mfs_{\mathrm{fil}}(\dv) \to \Sch/\mathbb K;  \mfs_{\mathrm{fil}}(\dv) (T) \mapsto T$$
where $\mfs_{\mathrm{fil}}(\dv) (T)$ is the groupoid of all flat families of filtered modules of fixed dimension 
vector parametrised by the scheme $T$ and morphism is given by isomorphism of families of modules. The 
morphism in this groupoid can be described as pair $(f : T' \to T, \phi : M' \xrightarrow{\simeq} M)$, where 
$\phi$ is an isomorphism of two flat families of filtered modules $M$ and $M'$. 
}
\end{definition}

Now, using the notion of stability via slope function, we get two open sub-stacks of $\mfs_{\mathrm{fil}}(\dv)$, 
denoted as $\mffs (\dv)$ and $\mfs_{\mathrm{fil}}^{\mathrm{s}}(\dv)$.  We will also denote by 
$\mathbf M^{\mathrm{ss}}_{\mathrm{fil}}(\dv)$ (respectively, $\mathbf M^{\mathrm{s}}_{\mathrm{fil}}(\dv)$) 
the coarse moduli space of semistable (respectively, stable)  filtered moduli stack. 

Now, we can lift the embedding of $\mathbf M_X^{\mathrm{ss}}(\tau_P)$ in $\mathbf M^{\mathrm{ss}}_{\mathrm{fil}}(\dv)$ 
as $1$-morphisms at the level of stacks \cite[\href{https://stacks.math.columbia.edu/tag/02XS}{Definition 02XS}]{stacks-project}.
 
\begin{proposition} \label{prop:stacky_adjunction}
There exists a morphism of moduli stacks
 $
 \Psi :  \mfs_X^{reg}(\tau_P) \to \mfs_{\mathrm{fil}}(\dv), 
 $
 i.e. $P_A \circ \Psi = P_X$. 
 Moreover, $\Psi$ is fully-faithful functor.
 \end{proposition}
 \begin{proof}
By Proposition \ref{s2-prop-eqiv-familiy},  we have an embedding of the category of flat familiy of 
$n$-regular parabolic sheaves having parabolic type $\tau_p$ into the category of flat family of filtered 
Kronecker modules having dimension vector $\dv$, where $\dv$ is determined by the fixed parabolic 
type $\tau_p$.

If we have any flat family of parabolic sheaves on $X \times S$ over a scheme $S$, then we get a 
filtered locally free sheaf on $S$,using regularity, which has fiberwise $A$-module structure. 
This defines a functor at the object level, which will be denoted by the same symbol $\Psi$, 
similar to \cite{AK07}.

If we have an isomorphism of flat families of parabolic sheaves over a fixed base scheme, then it will also
induce an isomorphism of flat families of filtered Kronecker modules over the same base scheme. Hence,
using the functoriality, we can extend the functor to an embedding of the groupoid 
$$\Psi_S :  \mfs_X^\mathrm{reg}(\tau_P)(S) \to  \mfs_{\mathrm{fil}}(\dv) (S); E \mapsto \Psi_S(E). $$
 
If we have a morphism between the base schemes $f : S \to S'$, then we can get a functor 
$(f \times 1)^* :  \mfs_X^\mathrm{reg}(\tau_P)(S') \to  \mfs_X^\mathrm{reg}(\tau_P)(S)$. Similar map exists between the 
fiber groupoids at module side.
 
Using the fact that cohomology commutes with flat base change, we can get the morphism
between moduli stacks $\Psi :  \mfs_X^\mathrm{reg}(\tau_P) \to  \mfs_{\mathrm{fil}}(\dv) $.
Now, the result \cite[\href{https://stacks.math.columbia.edu/tag/003Z}{Lemma 003Z}]{stacks-project} 
proves that the functor $\Psi$ is fully faithful.
\end{proof}
 
  \begin{lemma} \label{lemma:stacky_FM}
 There is a commutative diagram with horizontal maps being locally closed immersion on topological space of closed points
 $$
 \xymatrix{
  \mathscr M_X^{\mathrm{ss}}(\tau_P)   \ar[r] \ar[d] & \mffs (\dv) \ar[d] \\
    \mathbf M_X^{\mathrm{ss}}(\tau_P) \ar[r]  & \mathbf M_{\mathrm{fil}}^{\mathrm{ss}} (\dv)
 }
 $$
 and there is a similar diagram with stable in place of semistable loci and lower horizontal arrow is locally closed 
 scheme theoretic embedding. 
 \end{lemma} 
 \begin{proof}
 The morphism of stack on top is obtained by restricting the functor $\Psi$. 
 We will use the same notations for the category fibered in sets given by the moduli functors of last section. 
 Hence, we get the following canonical commutative diagram 
 $$
 \xymatrix{
  \mathscr M_X^{\mathrm{ss}}(\tau_P)   \ar[r] \ar[d] & \mffs (\dv) \ar[d] \\
\mf(\tau_p)_X^\mathrm{ss}   \ar[r]  & \mf_{\mathrm{fil}}^{\mathrm{ss}} (\dv)
}
$$
Now, using the commutative square coming from the Theorem \ref{main_thm:FMC}, we get the required 
commutative square. 
Since both $Q^{[\mathrm{ss}]} \ra \mathbf M_X^{\mathrm{ss}}(\tau_P)$  and 
$R^{\mathrm{ss}}_{\mathrm{fil}}(\dv) \ra \mathbf M_{\mathrm{fil}}^{\mathrm{ss}} (\dv)$ are good quotients 
with the same $S$-equivlalence classes, the last assertion on closed points follows. 

We can get the commutative square on stable locus by similar arguement. Since $Q^{[\mathrm{s}]}$ is 
closed equivariant subscheme of $R^{\mathrm{s}}_{\mathrm{fil}}(\dv)$ which is a principal bundle over 
$M_{\mathrm{fil}}^{\mathrm{s}} (\dv)$, we get the required scheme theoretic embedding 
(see \cite[Proposition 6.7]{AK07}).
\end{proof}  
\begin{proposition} \label{prop:stacky_square}
 There is a commutative diagram of $1$-morphisms between the stacks i.e. following diagram which commutes 
 up to natural isomorphism (or $2$-morphism)
 $$
 \xymatrix{ 
  [Q^{[\mathrm{ss}]}/G(\KAl)] \ar[r]^{\tilde{\mathbf{\iota}}}  \ar[d]^{\tilde{\mathbf{g}}} & [R^{\mathrm{ss}}_{\mathrm{fil}} / G(\KAl)] \ar[d]^{\tilde{\mathbf{h}}}  \\
  \mathscr M_X^{\mathrm{ss}}(\tau_P)   \ar[r]^{\Psi}  & \mffs (\dv) 
 }
 $$ 
 such that the vertical morphisms are isomrophisms. A similar assertion holds for the stable in place of semistable.
 \end{proposition}
 \begin{proof}
 Using the canonical inclusion of $Q^{[\mathrm{ss}]} \hookrightarrow R^{\mathrm{ss}}$, we get the horizontal 
 $1$-morphism of stacks. The bottom horizontal map is defined using the previous Proposition 
 \ref{prop:stacky_adjunction}. The vertical morphisms are defined using the local universal family of modules, 
 exactly similar to the maps $\tilde{g}$ and $\tilde{h}$, say $\tilde{\mathbf{g}}$ and $\tilde{\mathbf{h}}$, 
 respectively. The vertical maps are isomorphism follow from \cite[Proposition 3.3]{Go01}.
 
 We can check that the square commutes up to natural isomorphism given by the unit of the adjunction. 
 The assertion for the stable case follows using similar arguments.
 \end{proof}
 Now we get the following assertion from the description given in \cite[Example 2.29]{Go01} and properties of 
 GIT quotient from the Mumford's GIT.
 
 \begin{lemma} \label{lemma:GIT_square} 
 There exists a pullback diagram with vertical maps being canonical map to coarse moduli spaces with 
 horizontal maps being (set-theoretic) immerisions between stacks and schemes respectively.
 $$
 \xymatrix{
   [Q^{[\mathrm{ss}]}/G(\KAl)] \ar[r]  \ar[d] & [R^{\mathrm{ss}}_{\mathrm{fil}} / G(\KAl)] \ar[d] \\
       Q^{[\mathrm{ss}]} \git G  \ar@{^{(}->}[r]^j  & R^{\mathrm{ss}}_{\mathrm{fil}} \git G  
 }
 $$ 
 and there is a similar diagram for stable locus with the lower horizontal map being scheme theoretic immersion. 
 \end{lemma}
 
 \begin{proposition}
 There is a commutative diagram, up to natural isomorphism, where vertical maps are canonical map to 
 coarse moduli spaces
  $$
 \xymatrix{
    &  \mathscr M_X^{\mathrm{ss}}(\tau_P)  \ar[rr] \ar[dd] & & \mffs (\dv) \ar[dd]   & \\
  [Q^{[\mathrm{ss}]}/G(\KAl)]  \ar[ur]^{\simeq} \ar[dd] \ar[rr] & & [R^{\mathrm{ss}}_{\mathrm{fil}} / G(\KAl)] \ar[ur]^{\simeq} \ar[dd] & & \\
    & \mathbf M_X^{\mathrm{ss}}(\tau_P) \ar[rr]  & & \mathbf M_{\mathrm{fil}}^{\mathrm{ss}} (\dv) & \\
    Q^{[\mathrm{ss}]} \git G  \ar[ur] \ar[rr] & & R^{\mathrm{ss}}_{\mathrm{fil}} \git G  \ar[ur] &  & 
 }
 $$
 We get the similar diagram for stable locus with the vertical maps in bottom square being isomorphisms.
 \end{proposition}
 \begin{proof} 
 The two vertical squares come from the canonical maps coming from the definition of coarse moduli spaces 
 of respective stacks.
 Now, using previous Proposition \ref{prop:stacky_square}, we get that the top horizontal square will commute 
 up to natural isomorphism and the bottom horizontal square commutes using the Theorem \ref{main_thm:FMC}.

 Next, the front vertical square commutes using the Lemma \ref{lemma:GIT_square} and the back vertical square 
 commutes using the Lemma \ref{lemma:stacky_FM}. 
 The commutative diagram for the stable in place of semistable follows from similar arguments.
 \end{proof}
 
\begin{remark}\rm{
In the stable case, the front square in the above diagram becomes isomorphic to the square in the background, 
and hence we can deduce properties of the morphism between moduli stacks and coarse moduli spaces in 
the background square using the properties of morphism between quotient stacks diagram.
}
\end{remark}

\subsection*{Acknowledgement} 
We thank V. Balaji for asking the question of functorial moduli construction for 
parabolic sheaves and D. S. Nagaraj and N. Raghavendra for discussions and comments.
The first-named author would like to acknowledge the support of SERB-DST (India) 
grant number MTR/2018/000475. The second name author would like to acknowledge the support 
of DAE and DST INSPIRE grant number IFA 13-MA-25. We thank HRI and IIT Gandhinagar for providing great infrastructure.

We also thank the anonymous referee for helpful comments and corrections to improve the results.



\begin{thebibliography}{012345}
 \bibitem[AK07]{AK07} L.~\'{A}lvarez-C\'{o}nsul,~A.~King, \emph{A functorial
         construction of moduli of sheaves}, Invent. math. \textbf{168} (2007), 613-666.
 \bibitem[AD15]{AD15} S.~Amrutiya, U.~Dubey, \emph{Moduli of equivariant sheaves 
 		 and Kronecker-McKay modules}, Internat. J. Math. \textbf{26} (2015) 1550092
 		 (38 pp).
 \bibitem[AD20]{AD20}S.~Amrutiya,U.~Dubey, \emph{Moduli of filtered quiver representations}, Bull. Sci. Math \textbf{164} (2020) 102899, (29 pp).
 \bibitem[An09]{An09} Y.~Andr\'e, \emph{Slope filtrations}, Confluentes Math., \textbf{1}, no.1 (2009) 1-85.
 \bibitem[Bh92]{Bh92} U.~Bhosle, \emph{Generalised parabolic bundles and applications to torsionfree sheaves on nodal curves}, Ark. Mat. \textbf{30}, no. 2, (1992), 187?215. 
 \bibitem[Bh96]{Bh96} U.~Bhosle, \emph{Generalized parabolic bundles and applications. II},
Proc. Indian Acad. Sci. Math. Sci. \textbf{106}, no. 4, (1996) 403?420.
 \bibitem[Bi97]{Bi97} I.~Biswas, \emph{Parabolic bundles as orbifold bundles}, 
 		 Duke Math. J. \textbf{88} (2) (1997) 305-325. 
 \bibitem[Bo07]{Bo07} N.~Borne, \emph{Fibr\'es paraboliques et champ des racines}, Int. Math. Res. Not. 
 	 (2007), Article ID rnm049.
 \bibitem[BV12]{BV12} N.~Borne, A.~Vistoli, \emph{Parabolic sheaves on 
 		 logarithmic schemes}, Adv. Math. \text{231} (2012), 1327-1363. 
 \bibitem[Ch08]{Ch08} C.~Chindris, \emph{On GIT-fans for quivers}, preprint available at 
 	 \url{arXiv:0805.1440v1} [math.RT], 2008. 
 \bibitem[Ki94]{Ki94} A.~King, \emph{Moduli of representations of finite dimensional algebras},
         Quart. J. Math. Oxford Ser. \textbf{45} (1994), 180, 515-530.
 \bibitem[Go01]{Go01} Tom\'{a}s ~L ~G\'{o}mez, \emph{Algebraic stacks}, Proc. Indian Acad. Sci. (Math. Sci.),
         \textbf{111}, no. 1 (2001), 1-31.       
 \bibitem[GRT]{GRT} D.~Greb, J.~Ross and M.~Toma, \emph{Variation of Gieseker moduli spaces 
 		 via quiver GIT}, Geom. Topol. \textbf{20}, 1539-1610 (2016).
 \bibitem[HL97]{HL97} D.~Huybrechts,~M.~Lehn, \emph{The geometry of moduli spaces of sheaves},
         Vieweg-Braunschweig (1997).
 \bibitem[In00]{In00} M.~Inaba, \emph{Moduli of parabolic stable sheaves on a projective scheme}, 
 	J. Math. Kyoto Univ. \textbf{40} (2000), 119-136.
 \bibitem[La75]{La75} S.~G.~Langton,\emph{Valuative criteria for families of vector bundles on
         algebraic varieties}, Ann. Math. \textbf{101}, 88-110 (1975).
 \bibitem[La04]{La04} A.~Langer,\emph{Semistable sheaves in positive characteristic},
         Ann. Math. \textbf{159}, 251-276 (2004).
 \bibitem[MS80]{MS80} V.~B.~Mehta,~C.~S.~Seshadri, \emph{Moduli of vector 
         bundles on curves with parabolic structure}, Math. Ann.  \textbf{248} (1980), 205-239. 
 \bibitem[MY92]{MY92} M.~Maruyama,~K.~Yokogawa, \emph{Moduli of parabolic stable sheaves}, Math. Ann.
 	     \textbf{293} (1992), 77-99.
 \bibitem[Sc11]{Sc11} D.~Schl\"uter, \emph{Universal Moduli of Parabolic Sheaves on Stable Marked Curves.}
 		 (Ph. D. Thesis) Oxford University, UK, 2011.
 \bibitem[Si94]{Si94} C.~Simpson, \emph{Moduli of representations of the fundamental group
         of a smooth projective variety - I}, Inst. Hautes \'Etudes Sci. Publ. Math. \textbf{79},
         47-129 (1994).  
 \bibitem[Stacks20]{stacks-project} Stacks project authors, \emph{The Stacks project}, \url{https://stacks.math.columbia.edu} (2020).        
 \bibitem[Yo93]{Yo93} K.~Yokogawa, \emph{Compactification of moduli of parabolic sheaves and 
 	  moduli of parabolic Higgs sheaves}. J. Math. Kyoto Univ. \textbf{33} (1993), 451-504.
 \bibitem[Yo95]{Yo95} K.~Yokogawa, \emph{Infinitesimal deformation of 
         parabolic Higgs sheaves}, Internat. J. Math., \textbf{6} (1995) 125-148.  
\end{thebibliography}
\end{document}